\documentclass{amsart}
\usepackage{cleveref,float,amssymb,tikz}
\usetikzlibrary{arrows}

\newtheorem{theorem}{Theorem}[section]

\newtheorem{lemma}[theorem]{Lemma}
\newtheorem{proposition}[theorem]{Proposition}

\theoremstyle{definition}
\newtheorem{definition}[theorem]{Definition}
\theoremstyle{remark}
\newtheorem{remark}[theorem]{Remark}

\newcommand{\ch}{\text{\rm char}}

\title[Graded identities and Specht Property for $\mathrm{UT}_3^{(-)}$]{Graded polynomial identities and Specht property for the Lie algebra of upper triangular matrices of order 3}
\author{Daniela Martinez Correa}
\address{Department of Mathematics, Instituto de Matem\'atica e Estat\'istica, Universidade de S\~ao Paulo, SP, Brazil}
\email{danielam.correa@ime.usp.br}
\thanks{The first named author is supported by Fapesp, grant no.~2024/01338-0}

\author{Felipe Yukihide Yasumura}
\address{Department of Mathematics, Instituto de Matem\'atica e Estat\'istica, Universidade de S\~ao Paulo, SP, Brazil}
\email{fyyasumura@ime.usp.br}

\begin{document}
\begin{abstract}
We compute the graded polynomial identities for the variety of graded algebras generated by the Lie algebra of upper triangular matrices of order 3 over an arbitrary field and endowed with an elementary grading. We investigate the Specht property for the same family of varieties.
\end{abstract}
\maketitle

\section{Introduction}
This paper concerns the quantitative description of graded T-ideals. It is known that an explicit description of the T-ideal of polynomial identities satisfied by finite-dimensional simple associative algebras is available only for $2 \times 2$ matrices over an infinite field (see \cite{Dr81,K02,Raz73}). For finite fields, a description is known for $k \times k$ matrices, where $1 \le k \le 4$ (see \cite{MalKuz78,Gen81,GenSid82}). On the other hand, the T-ideal of polynomial identities for the algebra of upper triangular matrices with entries from any infinite field, viewed either as an associative or a Lie algebra, is well understood, as well as a basis for its relatively free algebra when the base field is infinite (see, for example, \cite{bahturin,Drenskybook,Siderov}).

When a group grading is imposed, the description of graded polynomial identities for the associative algebra of upper triangular matrices is known (see \cite{VinKoVa2004,GRiva}). Even if the algebra is taken to be infinite-dimensional, then the graded polynomial identities is known \cite{GY}. However, the problem of describing the graded polynomial identities for the same algebra, endowed with a non-associative product, is considerably more difficult, with only partial results available (see, for instance, \cite{CMa,DimasSa,GSK,KMa,PdrManu,Y23}). A similar difficulty arises when considering the polynomial identities of the algebra with the associative product but equipped with an additional structure, such as an involution \cite{VKS,UruG}, a superinvolution \cite{IMa}, a derivation \cite{VinNar,GiamRizzo,Nar}, a Hopf action \cite{CY}, or a combination of an involution and a grading \cite{DiogoGaldino,MelloY}.

To contribute to the general understanding of the graded polynomial identities of the algebra of upper triangular matrices viewed as a Lie algebra, we will investigate the graded polynomial identities for $3 \times 3$ upper triangular matrices over an arbitrary field (\Cref{universalbaseidentidades,almostuniversalbaseidentidades,almoscanonicallbaseidentidades,baseremaininggrading}), thereby extending the results from \cite{Y23}. Additionally, we will explore the Specht property of these same algebras (\Cref{specthUniversal,specthalmostUniversal,specthalmostcanonical,remainingSpecht} and \Cref{almostcanonical_nonspecht}), continuing the theory developed in \cite{correa} and complementing the results of \cite{CD}.

\section{Preliminaries}
All the algebras and vector spaces in this paper shall be over an arbitrary field $\mathbb{K}$. We denote by $G$ an abelian group with multiplicative notation and neutral element $1$.

 \subsection{Elementary Gradings on $\mathrm{UT}_3^{(-)}$}
Let $\mathrm{UT}_n^{(-)}$ be the Lie algebras of upper triangular matrices of order $n$ over $\mathbb{K}$. The isomorphism classes of elementary grandings on $\mathrm{UT}_n^{(-)}$ are described in \cite{KY2025}. Considering such classification, we list all equivalence classes of group gradings on $\mathrm{UT}_3^{(-)}$. For this, we need a notation. The matrix
\begin{center}
 $\begin{pmatrix} g & k \\
 & h\end{pmatrix}$  
\end{center}
shall denote the elementary grading on $\mathrm{UT}_3^{(-)}$ such that  $\deg e_{12}=g$, $\deg e_{23}=h$ and $\deg e_{13}=k$. It is clear that $\deg e_{ii}=1$ for $i\in\{1,2,3\}$ and necessarily $gh=k$. The given elementary grading will be denoted by $\Gamma(g,h)$. It is known that $\Gamma(g,h)\cong\Gamma(g',h')$ if and only if $g=g'$ and $h=h'$ or $g=h'$ and $h=g'$.

Let $e_1$ and $e_2$ be generators of a free abelian group, and consider the grading $\Gamma(e_1,e_2)$ on the algebra $\mathrm{UT}_3^{(-)}$ (known as the universal elementary grading). Given any elementary $G$-grading $\Gamma(g,h)$ on $\mathrm{UT}_3^{(-)}$, then the map $\langle e_1,e_2\rangle\to G$, such that $e_1\mapsto g$ and $h\mapsto h$ induces a coarsening of $\Gamma(e_1,e_2)$ that is isomorphic to $\Gamma(g,h)$. Hence, the elementary gradings on $\mathrm{UT}_3^{(-)}$ are obtained by a coarsening of $\Gamma(e_1,e_2)$.

The equivalence classes of elementary group grading on $\mathrm{UT}_3^{(-)}$ are listed below:
\begin{center}
\begin{tabular}{|p{3cm}|p{3cm}|p{3cm}|}
\hline
\multicolumn{3}{|p{9cm}|}{\centering Universal:\\$\displaystyle\left(\begin{array}{cc}g&k\\&h\end{array}\right),\quad g\ne h,\,1\notin\{g,h,k\}$}\\\hline
\centering Canonical:\\$\displaystyle\left(\begin{array}{cc}g&k\\&g\end{array}\right),\,k\ne1$  & {\centering Almost Universal:\\ $\displaystyle\left(\begin{array}{cc}g&1\\&g^{-1}\end{array}\right)$} & {\centering Remaining:\\ $\displaystyle\left(\begin{array}{cc}g&g\\&1\end{array}\right),\,g\ne1$}\\\cline{1-2}
\multicolumn{2}{|p{6cm}|}{\centering Almost Canonical: $\displaystyle\left(\begin{array}{cc}g&1\\&g\end{array}\right),\,g\ne1$}&\\\hline
\multicolumn{3}{|p{9cm}|}{\centering Trivial: $\displaystyle\left(\begin{array}{cc}1&1\\&1\end{array}\right)$}\\\hline
\end{tabular}
\end{center}

\subsection{Partially well-ordered sets}
We state some known results about partially well-ordered sets. Recall that a partially well-ordered set is a POSET satisfying the finite basis property. For more details about the following results, we refer the reader to \cite[Section 5.2]{bahturin}.
Consider the usual order $\leq$ on the non-negative integers $\mathbb{N}_0$ and let $m\geq 1$. By \cite[  Corollary 5.11]{bahturin}, 
\[\mathbb{N}_0^m=\underbrace{\mathbb{N}_0\times\cdots\times \mathbb{N}_0}_\text{$m$ factors}\]
is a partially well-ordered set with the  order $\leq_m$ given by  
\[(n_1, \ldots, n_m)\leq _m(n_1',\ldots, n_m') \ \ \mbox{if and only if} \ \ 
n_i\leq n_i'\] 
for every $i\in \{1,\ldots,m\}$. We have the following useful result.

\begin{proposition}[{\cite[Theorem 4.2]{Higman}}]\label{Sequenceorder}
Let $D(\mathbb{N}_0^m)$ be the set of all finite sequences $(a_1, ..., a_r)$ such that $a_1, ..., a_r\in \mathbb{N}_0^m$ and $r\geq 1$. Define the following order $\preceq$ on $D(\mathbb{N}_0^m)$: $(a_1,\ldots, a_r)\preceq (b_1,\ldots, b_s)$ if and only if there exists an injection $\psi:\mathbb{N}\to\mathbb{N}$ such that:  
\begin{itemize}
    \item[(i)] $\psi$ preserves the order, that is, if $u\leq v$ then $\psi(u)\leq \psi(v)$;
    \item[(ii)] $\psi(r)\leq s$;
    \item[(iii)] $a_i\leq_m b_{\psi(i)}$ for every $i\in \{1,\ldots,r\}$.
\end{itemize}
Then $(D(\mathbb{N}_0^m), \preceq)$ is a partially well-ordered set.
\end{proposition}

\subsection{Model for relatively free algebra}
Let $\mathcal{A}$ be a finite-dimensional $\Omega$-algebra with basis $\{a_1,\ldots,a_n\}$, and consider the generic algebra $\mathcal{R}$, generated by the elements
$$
r_k:=\sum_{i=1}^n\xi_i^{(k)}\otimes a_i\in\mathbb{K}[\Xi]\otimes_\mathbb{K}\mathcal{A},
$$
where $\Xi=\{\xi_i^{(k)}\mid k\in\mathbb{N},i\in\{1,\ldots,n\}\}$ and $\mathbb{K}[\Xi]$ is the polynomial $\mathbb{K}$-algebra in the commutative and associative variables $\Xi$. It is known that, if the base field is infinite, then $\mathcal{R}$ is the relatively free algebra in the variety generated by $\mathcal{A}$. To include the case where the base field is finite, we proceed with the following construction. Let $\mathcal{F}=\mathrm{Func}(\mathbb{K}^\Xi,\mathbb{K})$ denote the set of all maps $\mathbb{K}^\Xi\to\mathbb{K}$. Then, the ring structure of $\mathbb{K}$ turns $\mathcal{F}$ into a ring, via pointwise sum and multiplication. We have an algebra homomorphism $\mathbb{K}[\Xi]\to\mathcal{F}$ by the following: each element of $\mathbb{K}^\Xi$ is a map $\varphi_0:\Xi\to\mathbb{K}$. Hence, it induces an algebra homomorphism $\overline{\varphi_0}:\mathbb{K}[\Xi]\to\mathbb{K}$. So, we have a map $\iota:\mathbb{K}[\Xi]\to\mathcal{F}$ such that $\iota(f)(\varphi_0)=\overline{\varphi_0}(f)$. Note that $\iota$ is injective if and only if $\mathbb{K}$ is infinite. Then, we have an algebra homomorphism $\iota\otimes1:\mathbb{K}[\Xi]\otimes_\mathbb{K}\mathcal{A}\to\mathcal{F}\otimes_\mathbb{K}\mathcal{A}$. Therefore, we have the following diagram:
\begin{center}
\begin{tikzpicture}
    \node (1) at (0,0) {$\mathcal{R}$};
    \node (11) at (4,0) {$\mathcal{R}'$};
    \node (22) at (4,2) {$\mathrm{Func}(\mathbb{K}^\Xi,\mathbb{K})\otimes_\mathbb{K}\mathcal{A}$};
    \node (2) at (0,2) {$\mathbb{K}[\Xi]\otimes_\mathbb{K}\mathcal{A}$};
    \draw[->] (1) -- node[right] {} (11);
    \draw[right hook->] (1) -- node[above] {} (2);
    \draw[->] (2) -- node[right] {} (22);
    \draw[right hook->] (11) -- node[above] {} (22);
\end{tikzpicture}
\end{center}
As in the classical case, we have:
\begin{proposition}
$\mathcal{R}'$ is the relatively free algebra of the variety generated by $\mathcal{A}$, where the base field $\mathbb{K}$ is arbitrary (finite or infinite).
\end{proposition}
As mentioned before, if $\mathbb{K}$ is infinite, then $\mathcal{R}'\cong\mathcal{R}$ so there is nothing new in this construction.

We shall denote the image of each variable $\xi_i^{(k)}$ under $\iota$ by the same letter $\xi_i^{(k)}$, where the second is viewed as a polynomial map $\mathbb{K}^\Xi\to\mathbb{K}$. Hence, if $\mathbb{K}$ is finite containing $q$ elements, then $\left(\xi_i^{(k)}\right)^q=\xi_i^{(k)}$.

Now, assume that $\mathfrak{g}$ is a finite-dimensional Lie algebra with basis $\{a_1,\ldots,a_n\}$ and there exists a nonzero element $z\in\mathfrak{z}(\mathfrak{g})$ that can be written as $z=a_1+\sum_{i=2}^n\alpha_ia_i$. Then, a polynomial $f(x_1,\ldots,x_m)\in\mathcal{L}(X)$ which is a sum of Lie words of length at least $2$, is a polynomial identity for $\mathfrak{g}$ if and only if
$$
f(r_1-\xi_1^{(1)}z,\ldots,r_m-\xi_1^{(m)}z)=0.
$$
Hence, the generic elements
$$
r_k':=\sum_{i=2}^n\xi_i^{(k)}\otimes a_i\in\mathrm{Func}(\mathbb{K}^\Xi,\mathbb{K})\otimes\mathfrak{g}
$$
are enough to describe the polynomial identities of $\mathfrak{g}$.

Conversely, let $\psi:\mathcal{L}(X)\to\mathfrak{g}$ be a homomorphism and assume that $\psi(x_k)=\sum_{i=1}^n\alpha_i^{(k)}a_i$, for each $k\in\mathbb{N}$. Define $\psi':\mathcal{L}(X)\to\mathfrak{g}$ via
$$
\psi'(x_k)=\psi(x_k^{(g)})-a_1^{(k)}z,\quad k\in\mathbb{N}.
$$
Then, $\psi(m)=\psi'(m)$ for each commutator $m$ of length at least $2$, and $\psi'(x_k)=\sum_{i=2}^na_i^{(k)\prime}a_i$, for each $k\in\mathbb{N}$. Note that a similar construction holds valid for $G$-graded algebras.

We register the previous discussion for the particular case of upper triangular matrices.
\begin{remark}\label{evaluacionesUT}
Let $\psi:\mathcal{L}(X)\to\mathrm{UT}_n^{(-)}$ be a homomorphism. Then, there exists a homomorphism $\psi':\mathcal{L}(X)\to\mathrm{UT}_n^{(-)}$ such that $\psi(m)=\psi'(m)$, for each commutator $m$ of length at least $2$, and, for each $x\in X$, the entry $(n,n)$ of $\psi'(x)$ is $0$.
\end{remark}

\subsection{Transfer of infinite basis property} The following result is an easy remark but will be important for our purposes:
\begin{lemma}\label{transferSpecht}
Let $\mathcal{A}$ be an $\Omega$-algebra, $\Gamma$ a finite $G$-grading on $\mathcal{A}$, $\alpha:G\to H$ a group homomorphism and $\Gamma^\alpha$ the $H$-grading on $\mathcal{A}$ induced by $\alpha$. Let $\pi:\mathbb{K}\langle X^H\rangle\to\mathbb{K}\langle X^G\rangle$ be defined by
$$
\pi(x_{i,h})=\sum_{g\in\alpha^{-1}(h)\cap\mathrm{Supp}\,\Gamma}x_{i,g}.
$$
If there exists a family $\{f_i\}_{i\in\mathbb{N}}\subseteq\mathbb{K}\langle X^H\rangle$ such that
$$
\langle\pi(f_1)\rangle_{T_G}+\mathrm{Id}(\Gamma)\subsetneq\langle\pi(f_1),\pi(f_2)\rangle_{T_G}+\mathrm{Id}(\Gamma)\subsetneq\cdots,
$$
Then $\Gamma^\alpha$ does not satisfy the Specht property. More precisely,
$$
\langle f_1\rangle_{T_H}+\mathrm{Id}(\Gamma^\alpha)\subsetneq\langle f_1,f_2\rangle_{T_H}+\mathrm{Id}(\Gamma^\alpha)\subsetneq\cdots.
$$
\end{lemma}
\begin{proof}
It follows from the fact that $\pi(\langle f\rangle_{T_H})\subseteq\langle\pi(f)\rangle_{T_G}$ and $\pi(\mathrm{Id}(\Gamma^\alpha))\subseteq\mathrm{Id}(\Gamma)$.
\end{proof}

\subsection{Dual in a relatively free algebra} We shall describe the inverse process of the classical method of determining a basis of a relatively free algebra.
\begin{lemma}\label{inverseprocess}
Let $\mathbb{K}$ be an arbitrary field (finite or infinite), $\mathcal{A}$ an $\Omega$-algebra and $S=\{g_1,\ldots,g_m\}\subseteq\mathbb{K}_\Omega\langle X\rangle$ be a finite $\mathbb{K}$-linearly independent set, modulo $\mathrm{Id}(\mathcal{A})$. Then, there exists a finite family of evaluations $\{\psi_{j}:\mathbb{K}_\Omega\langle X\rangle\to\mathcal{A}\}_{j\in\Lambda}$ and, for each $\ell\in\{1,\ldots,m\}$, a finite family of linear maps $\{f_{j\ell}:\mathcal{A}\to\mathbb{K}\}_{j\in\Lambda}$ such that
$$
\sum_{j\in\Lambda}f_{j\ell}\circ\psi_j(g_i)=\delta_{i\ell},\quad\forall i\in\Lambda.
$$
\end{lemma}
\begin{proof}
The set $\{g_1+\mathrm{Id}(\mathcal{A}),\ldots,g_m+\mathrm{Id}(\mathcal{A})\}$ is $\mathbb{K}$-linearly independent in $\mathbb{K}_\Omega\langle X\rangle/\mathrm{Id}(\mathcal{A})$. The latter is a subalgebra of the direct product of copies of $\mathcal{A}$ via the map $\mu:\mathbb{K}_\Omega\langle X\rangle/\mathrm{Id}(\mathcal{A})\to\prod_{H}\mathcal{A}$, defined by $h\mapsto\left(\psi(h)\right)_{\psi\in H}$,
where $H=\mathrm{Hom}(\mathbb{K}_\Omega\langle X\rangle,\mathcal{A})$. The space $\prod_H\mathcal{A}$ is the inverse limit of finite direct products of copies of $\mathcal{A}$, indexed by finite subsets of $H$. Since $S$ is finite, there exists a finite subset $\Lambda\subseteq H$ where the image of $S$ is $\mathbb{K}$-linearly independent in $\prod_{\Lambda}\mathcal{A}$. Denote $p:\prod_{H}\mathcal{A}\to\prod_{\Lambda}\mathcal{A}$ the projection. So, considering a dual set of $\{p\circ\mu(g_i+\mathrm{Id}(\mathcal{A}))\mid i\in\{1,\ldots,m\}\}$, we can find a linear functional $\varphi_\ell:\prod_{\Lambda}\mathcal{A}\to\mathbb{K}$ such that $\varphi_\ell(p\circ\mu(g_i+\mathrm{Id}(\mathcal{A})))=\delta_{i\ell}$, for each $\ell\in\{1,\ldots,m\}$. However, every linear map $\varphi_\ell$ is expressed as a sum $\sum_{j\in\Lambda}f_{j\ell}\circ\psi_j$, where each $f_{j\ell}:\mathcal{A}\to\mathbb{K}$ is linear.
\end{proof}

\section{Graded Identities and Specht Property for $\mathrm{UT}_3^{(-)}$}
We will compute the graded polynomial identities for each of the elementary group gradings introduced earlier, over an arbitrary field. Additionally, we will investigate the respective Specht property for each of these gradings when the base field is infinite. It is worth noting that these results are already known for the Trivial grading \cite{bahturin,Siderov} and the Canonical grading \cite{CGR, correa}.

\noindent\textbf{Notation.} We use the following notation, for $m\in\mathbb{N}$ and $x$, $y$ elements of a Lie algebra:
 $$
 [x,y^{(m)}]:=[x,\underbrace{y,\ldots,y}_{\text{$m$ times}}].
 $$
Let $G$ be an abelian group. We consider the free $G$-graded Lie algebra $\mathcal{L}(X^G)$, where $X^G=\{x_{i,g}\mid i\in\mathbb{N},g\in G\}$. The variables of trivial degree will be denoted by $y$, $y_1$, $y_2$, \dots, and $z$, $z_1$, $z_2$, \dots will be reserved for variables of non-trivial degree.

\subsection{Universal Grading}\label{universalgradingsection}
For this section, we let $\mathbb{K}$ be an arbitrary field and $\mathrm{UT}_3^{(-)}$ be endowed with the Universal grading $\Gamma(e_1,e_2)$, where $\langle e_1,e_2\rangle\cong\mathbb{Z}^2$. 

It is clear that the polynomials $[y_1,y_2]$ and $x_g$, where $g\notin\mathrm{Supp}\,\Gamma(e_1,e_2)$, are graded polynomial identities for $(\mathrm{UT}_3^{(-)},\Gamma(e_1,e_2))$. In addition, if $\mathbb{K}$ is finite with $q$ elements, then $[x_g,y^{(q)}]-[x_g,y]$, for each $g\in\{e_1,e_2,e_1e_2\}$, is a polynomial identity as well. The characterization of $\mathrm{Id}_G(\mathrm{UT}_3^{(-)},\Gamma(e_1,e_2))$ is as follows.
 
\begin{theorem}\label{universalbaseidentidades}
\renewcommand{\labelenumi}{(\roman{enumi})}
 Let $\mathbb{K}$ be an arbitrary field and $G=\langle e_1,e_2\rangle\cong\mathbb{Z}^2$. Then, the ideal $\mathrm{Id}_G(\mathrm{UT}_3^{(-)}(\mathbb{K}),\Gamma(e_1,e_2))$ is generated as $T_G$-ideal by:
 \begin{enumerate}
 \item $x_g$, $g\notin\mathrm{Supp}\,\Gamma(e_1,e_2)$,
  \item $[y_1,y_2]$.
 \end{enumerate}
 In addition, if $\mathbb{K}$ is a finite field containing $q$ elements, then we include
 \begin{enumerate}
 \setcounter{enumi}{2}
\item $[x_g,y^{(q)}]-[x_g,y]$, $g\in\{e_1,e_2,e_1e_2\}$.
 \end{enumerate}
 A basis of the relatively free algebra $\mathcal{L}(X^G)$, 
 modulo $\mathrm{Id}_G(\mathrm{UT}_3^{(-)},\Gamma(e_1,e_2))$, constitute of all the following polynomials:
   \begin{align}
    &[x_g,y_1^{(p_1)},\ldots,y_r^{(p_r)}],\label{universal.1}\\
    &[x_{e_1},y_1^{(p_1)},\ldots,y_r^{(p_r)},x_{e_2},y_1^{(q_1)},\ldots,y_r^{(q_r)}]\label{universal.2},
   \end{align}
   where $g\in \{e_1,e_2, e_1e_2\}$, $r\ge0$ and $0\le p_\ell,q_\ell<|\mathbb{K}|$, for each $\ell\in\{1,2,\ldots,r\}$.
 \end{theorem}
 \begin{proof}
Let $T$ be the $T_G$-ideal generated by the polynomials (i)-(iii) of the statement. First, we shall show that given $m\in\mathcal{L}(X^G)$, we can write $m$ as a linear combination of polynomial of type \eqref{universal.1} and \eqref{universal.2}, modulo $T$. For this, it is sufficient to prove the statement when
\[m=[x_{i_1,g_1},\ldots, x_{i_r,g_r}]\]
Notice that there are at most two variable of non-trival degree in $m$, modulo $T$. It follows that 
\[m=[y_{i_1},\ldots, y_{i_s}, x_{i, g}, y_{j_1},\ldots, y_{j_t}, x_{j, h}, y_{l_1},\ldots, y_{l_r}],\]
for some $g, h\in\{e_1,e_2,e_1e_2\}$ and the block formed by the variables $x_{j,h}, y_{l_1},\ldots, y_{l_r}$ can be empty. Since $[y_1,y_2]\in T$, we have that $s\leq 1$ and the variables $y$'s can be ordered in non-decreasing way. Hence, modulo $T$, one has
\[m=[x_{i, g}, y_{1}^{(p_1)},\ldots, y_{r}^{(p_r)}, x_{j, h}, y_{1}^{(q_1)},\ldots, y_{r}^{(q_r)}], \]
where $r\geq 0$, $p_l, q_l\geq 0$  for all $l\in\{1,\ldots, r\}$, $g,h\in\{e_1,e_2, e_1e_2\}$ and the second non-trivial variable  $x_{j,h}$ may not appear. Moreover, if $\mathbb{K}$ is finite having $q$ elements, then the identity  $[x_g,y^{(q)}]-[x_g,y]\in T$ tells us that $0\leq p_l, q_l< q$,  for all $l\in\{1,\ldots, r\}$. It follows that, modulo $T$,
\[m=[x_{i, g}, y_{1}^{(p_1)},\ldots, y_{r}^{(p_r)}, x_{j, h}, y_{1}^{(q_1)},\ldots, y_{r}^{(q_r)}], \]
where $r\geq 1$, $0\leq p_l, q_l< |\mathbb{K}|$  for all $l\in\{1,\ldots, r\}$ and the second non-trivial variable  $x_{j,h}$ may not appear.

If the variable $x_{j,h}$ does not appear in the polynomial $m$, we have that $m$ is a polynomial of type (1).

If the second variable appears in $m$, we have $g,h\in\{e_1, e_2\}$ with $g\neq h$. If $g=e_1$ and $h=e_2$, the polynomial $m$ is a polynomial of type (2).  If $g=e_2$ and $h=e_1$, using the Jacoby Identity on $m$ iteratively, we can write $m$ as a linear combination of polynomials of type (2). 
Now, we shall prove that the set  $M$ of polynomials of type (1) and type (2) is linearly indepedent, modulo $\mathrm{Id}_G(\mathrm{UT}_3^{(-)},\Gamma(e_1,e_2))$. In fact, suppose that 
\[f=\sum\limits_{h\in M}\alpha_h h \in \mathrm{Id}_G (\mathrm{UT}_3^{(-)},\Gamma(e_1,e_2)),\]
where $\alpha_h\in\mathbb{K}$. Without loss of generality, we can suppose either that 
\[f=\sum\alpha [x_g, y_{1}^{(p_1)},\ldots, y_{r}^{(p_r)}],\]
where $g\in \{e_1,e_2, e_1e_2\}$ and the sum runs over all possibilities of choices of $p_\ell$, where $0\leq p_{\ell}<|\mathbb{K}|,\,\, 1\leq l\leq r$; or 
\[f=\sum\limits_{s=1}^t \alpha_{s}[x_{e_1},y_1^{(p_{1,s})},\ldots,y_r^{(p_{r,s})},x_{e_2},y_1^{(q_{1,s})},\ldots,y_r^{(q_{r,s})}],\]
where $r\ge0$ and $0\le p_{\ell,s},q_{\ell,s}<|\mathbb{K}|$.

First, suppose 
$f=\sum\limits_{s=1}^t \alpha_{s}[x_{e_1},y_1^{(p_{1,s})},\ldots,y_r^{(p_{r,s})},x_{e_2},y_1^{(q_{1,s})},\ldots,y_r^{(q_{r,s})}]$ and consider the polynomial algebra $\mathbb{K}[\xi_{i,j}\mid i\in\{1,2,3\},\,  j\in\mathbb{N}]$. Evaluating as
\[x_{e_1}\mapsto e_{12},\quad x_{e_2}\mapsto e_{23},\quad y_k\mapsto\sum\limits_{l=2}^3 \xi_{l,k} e_{ll},\]
 in the polynomial $f$, we obtain 
 \[f\mapsto\left(\sum\limits_{s=1}^t \alpha_{s}\left(\prod\limits_{l=1}^r (\xi_{2,l})^{p_{l,s}}(\xi_{3,l})^{q_{l,s}}\right)\right) e_{13}.\]
Since $0\leq p_{l,s}, q_{l,s}<|\mathbb{K}|$, we conclude that $\alpha_s=0$ for $1\leq s\leq t$. 

If $f=\sum\alpha [x_g, y_{1}^{(p_1)},\ldots, y_{r}^{(p_r)}]\in \mathrm{Id}_G (\mathrm{UT}_3^{(-)},\Gamma(e_1,e_2)))$, then a similar evaluation gives $\alpha=0$, for each $\alpha$.
 \end{proof}
 
Next, we investigate the Specht property for the variety of graded Lie algebras generated by $\mathrm{UT}_3^{(-)}$ endowed with the Universal grading. First, we shall establish some notation.

For each commutator $f=[x_g, y_{1}^{(k_1)},\ldots, y_{n}^{(k_n)}]$, where $g\in\{e_1, e_2, e_1 e_2\}$, we associate the following finite sequence of non-negative integers
\begin{equation}
S_f= (k_1,\ldots, k_n).
\end{equation}
\begin{lemma}\label{oderSfSg}
 Let $(D(\mathbb{N}_0), \preceq)$ be the partially well-ordered set from 
\Cref{Sequenceorder}.  Consider the following commutators
\[f= [x_g, y_{1}^{(k_1)},\ldots, y_{n}^{(k_n)}],\]
\[h= [x_g, y_{1}^{(k'_1)},\ldots, y_{m}^{(k'_m)}],\]
where $g\in\{e_1,e_2, e_{1}e_2\}$ and $0\leq k_1,\ldots, k_n,k'_1,\ldots, k'_m<|\mathbb{K}|$. If $S_f\preceq S_h$, then $h\in \langle f\rangle^{T_G}+ \mathrm{Id}_G(\mathrm{UT}_3^{(-)},\Gamma(e_1,e_2))$.
\end{lemma}
\begin{proof}
Since $S_f\preceq   S_h$,  there exists a subsequence $(k'_{i_1}, \ldots, k'_{i_n})$ of $S_h$ such that $ k_j \leq k'_{i_j}$ for every $ j \in \{1, \ldots,n\}$. Then
    \[ \tilde{f}=f(x_{g}, y_{i_1},\ldots, y_{i_n})=[x_{g},y_{i_1}^{(k_1)}, \ldots, y_{i_n}^{(k_n)}] \in \langle f\rangle^{T_G}. \]
Substituting as $x_g \mapsto [x_g,y_{i_1}^{(k'_{i_1}-k_1)},\ldots, y_{i_n}^{(k'_{i_n}-k_n)},y_{j_1}^{(k'_{j_1})}, \ldots, y_{j_{m-n}}^{(k'_{j_{m-n}})} ]$ on $\tilde{f}$,
where $\{j_{1}, \ldots, j_{m-n}\}=\{1,\ldots,m\}\setminus \{i_{1}, \ldots, i_{n}\}$, we obtain a new polynomial $\overline{f}\in \left< f \right>^{T_G}$. By \Cref{universalbaseidentidades},
\begin{align*} [\overline{f},y_{i_{1}}^{(l'_{i_{1} }-l_1)},\ldots,y_{i_n}^{(l'_{i_n}- l_n)},y_{j_1}^{(l'_{j_1})}, \ldots, y_{j_{m-n}}^{(l'_{j_{m-n}})}]+ \mathrm{Id}_G(\mathrm{UT}_3^{(-)},\Gamma(e_1,e_2))=\\h+\mathrm{Id}_G(\mathrm{UT}_3^{(-)},\Gamma(e_1,e_2)).
\end{align*}
Therefore, $h\in \langle f\rangle^{T_G}+ \mathrm{Id}_G(\mathrm{UT}_3^{(-)},\Gamma(e_1,e_2))$.
\end{proof}

\begin{lemma}\label{finitecomutatoruniversal}
Let $g\in\{e_1, e_2, e_1 e_2\}$ and $C$ be a subset of the set 
\[\left\{[x_g, y_1^{(k_1)},\ldots, y_n^{(k_n)}]\mid n\geq 0,\, k_1,\ldots, k_n\geq 0\right\}.\]  
Then, there exist $f_1,\ldots, f_s\in C$ such that
\[ \langle f_1,\ldots, f_s\rangle ^{T_G}+ \mathrm{Id}_G(\mathrm{UT}_3^{(-)},\Gamma(e_1,e_2))= \langle C\rangle^{T_G}+\mathrm{Id}_G(\mathrm{UT}_3^{(-)},\Gamma(e_1,e_2)).\]
\end{lemma}
\begin{proof}
Let $D(\mathbb{N}_0, \preceq)$ be the partially well-ordered set from \Cref{Sequenceorder}. Consider the set $S_{C}=\{S_{f}\mid f\in C\}$. Note that $S_{C}\subseteq D(\mathbb{N}_0)$, so $S_C$ is partially well-ordered. It means there exists a finite set of minimal elements $\{S_{f_1},\ldots, S_{f_s}\}\subseteq S_{C}$. Thus, for each $S_f\in S_{C}$, there is $i\in \{1,\ldots, s\}$, such that $S_{f_i}\preceq S_{f}$. By \Cref{oderSfSg}, we obtain 
\[ \langle f_1,\ldots, f_s\rangle ^{T_G}+ \mathrm{Id}_G(\mathrm{UT}_3^{(-)},\Gamma(e_1,e_2))= \langle C\rangle^{T_G}+\mathrm{Id}_G(\mathrm{UT}_3^{(-)},\Gamma(e_1,e_2)).\]

\end{proof}

For each  $f=[x_{e_1},y_1^{(k_1)}, \ldots, y_n^{(k_n)}, x_{e_2}, y_1^{(l_1)}, \ldots, y_n^{(l_n)}]$, we associate the following finite sequence of elements in $\mathbb{N}_0^2$:
\begin{equation}
V_f= ( (k_1,l_1),\ldots,(k_n,l_n)).
\end{equation}
\begin{lemma}\label{oderseqfe1e2}
 Let $(D(\mathbb{N}_0^2), \preceq)$ be the partially well-ordered set from 
\Cref{Sequenceorder}.  Consider the following commutators 
\[f(x_{e_1}, x_{e_2}, y_1,\ldots, y_n)= [x_{e_1}, y_{1}^{(k_1)},\ldots, y_{n}^{(k_n)}, x_{e_2}, y_1^{(l_1)},\ldots, y_n^{(l_n)}],\]
\[h(x_{e_1}, x_{e_2},y_1,\ldots, y_m)= [x_{e_1}, y_{1}^{(k'_1)},\ldots, y_{m}^{(k'_m)}, x_{e_2}, y_1^{(l'_1)},\ldots, y_m^{(l'_m)}],\]
where $0\leq k_1,\ldots, k_n,l_1,\ldots, l_n,k'_1,\ldots, k'_m,l'_1,\ldots, l'_m$. If $V_f\preceq V_h$, then  $h\in \langle f\rangle^{T_G}+ \mathrm{Id}_G(\mathrm{UT}_3^{(-)},\Gamma(e_1,e_2))$
\end{lemma}
\begin{proof}
Since $V_f\preceq V_h$,  there exists a subsequence $((k'_{i_1}, l'_{i_1}), \ldots, (k'_{i_n},l'_{i_n}))$ of $S_h$ such that $(k_j,l_j) \le_2 (k'_{i_j},l'_{i_j})$ for every $ j \in \{1, \ldots,n\}$. Therefore $k_j \leq k'_{i_j}$ and $l_j\leq l'_{i_j}$  for every $j \in \{1,\ldots,n\}$. Then
    \[ \tilde{f}=f(x_{e_1}, x_{e_2}, y_{i_1},\ldots, y_{i_n})=[x_{e_1},y_{i_1}^{(k_1)}, \ldots, y_{i_n}^{(k_n)}, x_{e_2}, y_{i_1}^{(l_1)},\ldots, y_{i_n}^{(l_n)}] \in \left< f \right>^{T_G}. \]
Proceeding similarly as in the proof of \Cref{oderSfSg}, an adequate evaluation on $x_{e_1}$ and $x_{e_2}$ gives $h\in \langle f\rangle^{T_G}+ \mathrm{Id}_G(\mathrm{UT}_3^{(-)},\Gamma(e_1,e_2))$.
\end{proof}

\begin{definition}\label{orderlinearmonomios}
Consider the following commutators of the same multidegree
\[f(x_{e_1},x_{e_2}, y_1,\ldots, y_n)= [x_{e_1}, y_{1}^{(k_1)},\ldots, y_{n}^{(k_n)}, x_{e_2}, y_1^{(l_1)},\ldots, y_n^{(l_n)}],\]
\[h(x_{e_1}, x_{e_2}, y_1,\ldots, y_n)= [x_{e_1}, y_{1}^{(k'_1)},\ldots, y_{n}^{(k'_n)}, x_{e_2}, y_1^{(l'_1)},\ldots, y_n^{(l'_n)}],\]
where $0\leq k_1,\ldots, k_n,l_1,\ldots, l_n,k'_1,\ldots, k'_n,l'_1$. 
We say that $f<' h$ if either
\begin{itemize}
\item[(a)] $\sum\limits_{i=1}^n k_i< \sum\limits_{i=1}^n k'_i$, or
\item[(b)] $\sum\limits_{i=1}^n k_i= \sum\limits_{i=1}^n k'_i$ and $(k_1,\ldots, k_n)<_{l} (k'_1,\ldots, k'_n)$.
\end{itemize}
Here, $<_l$ is the lexicographic order on $\mathbb{N}_0^ n$.
\end{definition}

\begin{definition}\label{monomiolider}
Let $f$ be a multihomogeneous polynomial such that
\[f(x_{e_1}, x_{e_2}, y_1,\ldots, y_n)=\sum\limits_{s=1}^t \alpha_{s}[x_{e_1},y_1^{(p_{1,s})},\ldots,y_r^{(p_{r,s})},x_{e_2},y_1^{(q_{1,s})},\ldots,y_r^{(q_{r,s})}],\] 
where $\alpha_s\in\mathbb{K}\setminus\{0\}$ for $1\leq s\leq t$. The leading monomial of $f$ with respect to $<'$ is
\[\mathcal{M}(f)=\max_{<'}\{ [x_{e_1},y_1^{(p_{1,s})},\ldots,y_r^{(p_{r,s})},x_{e_2},y_1^{(q_{1,s})},\ldots,y_r^{(q_{r,s})}]\mid 1\leq s\leq t\}.\]
\end{definition}
\begin{lemma}\label{compativel}
Let $f$ and $h$ be multihomogeneous polynomials of the same multidegree and write
\[f(x_{e_1},x_{e_2}, y_1,\ldots, y_n)= [x_{e_1}, y_{1}^{(k_1)},\ldots, y_{n}^{(k_n)}, x_{e_2}, y_1^{(l_1)},\ldots, y_n^{(l_n)}],\]
\[h(x_{e_1}, x_{e_2}, y_1,\ldots, y_n)= [x_{e_1}, y_{1}^{(k'_1)},\ldots, y_{n}^{(k'_n)}, x_{e_2}, y_1^{(l'_1)},\ldots, y_n^{(l'_n)}].\]
If $f<' h$, then
\begin{enumerate}\renewcommand{\labelenumi}{(\roman{enumi})}
\item \begin{equation*}
\begin{split}
[x_{e_1}, y_{1}^{(k_1)},\ldots, y_{n}^{(k_n)}, y_{n+1}, x_{e_2}, y_1^{(l_1)},\ldots, y_n^{(l_n)}]& <'\\
  [x_{e_1}, y_{1}^{(k'_1)},\ldots, y_{n}^{(k'_n)}, y_{n+1}, x_{e_2}, y_1^{(l'_1)},\ldots, y_n^{(l'_n)}]
\end{split}\end{equation*}
\item \begin{equation*}
\begin{split}[x_{e_1}, y_{1}^{(k_1)},\ldots, y_{n}^{(k_n)}, x_{e_2}, y_1^{(l_1)},\ldots, y_n^{(l_n)}, y_{n+1}]&<'\\
[x_{e_1}, y_{1}^{(k'_1)},\ldots, y_{n}^{(k'_n)}, x_{e_2}, y_1^{(l'_1)},\ldots, y_n^{(l'_n)}, y_{n+1}]\end{split}\end{equation*}
    \item  \begin{equation*}\begin{split}
    [x_{e_1}, y_{1}^{(k_1)},\ldots, y_i^{(k_i+1)},\ldots,  y_{n}^{(k_n)}, x_{e_2}, y_1^{(l_1)},\ldots, y_n^{(l_n)}]<'&\\
    [x_{e_1}, y_{1}^{(k'_1)},\ldots, y_i^{(k'_i+1)},\ldots, y_{n}^{(k'_n)}, x_{e_2}, y_1^{(l'_1)},\ldots, y_n^{(l'_n)}]\end{split}\end{equation*}
    \item \begin{equation*}
    \begin{split}
    [x_{e_1}, y_{1}^{(k_1)},\ldots,  y_{n}^{(k_n)}, x_{e_2}, y_1^{(l_1)},\ldots, y_i^{(l_i+1)},\ldots, y_n^{(l_n)}]<'&\\
    [x_{e_1}, y_{1}^{(k'_1)},\ldots, y_{n}^{(k'_n)}, x_{e_2}, y_1^{(l'_1)},\ldots,  y_i^{(l'_i+1)},\ldots, y_n^{(l'_n)}]
    \end{split}
    \end{equation*}
\end{enumerate}
\end{lemma}
\begin{proof}
 The proof is straightforward.
\end{proof}
As a consequence of the previous lemma, let $f$ and $h$ be as above, with $f<'h$. Then $\psi(f)<'\psi(h)$ and $[f,y]<'[h,y]$, for any evaluation $\psi$ of the kind $x_{e_i}\mapsto[x_{e_i},y_1^{(r_1)},\ldots,y_n^{(r_n)}]$ and $y_i\mapsto y_{s_i}$ (with $s_1<s_2<\cdots$), and variable of trivial degree $y$. Hence, the leading monomial of the polynomial $[\psi(f),y_1^{(k_1)},\ldots,y_n^{(k_n)}]$ is $[\psi(\mathcal{M}(f)),y_1^{(k_1)},\ldots,y_n^{(k_n)}]$, for $\psi$ as before.

\begin{definition}\label{poltypee_1e_2}
We say that $f$ is polynomial of type $(e_1, e_2)$ if $f$ is multihomogeneous and
\[f(x_{e_1}, x_{e_2}, y_1,\ldots, y_n)=\sum\limits_{s=1}^t \alpha_{s}[x_{e_1},y_1^{(p_{1,s})},\ldots,y_r^{(p_{r,s})},x_{e_2},y_1^{(q_{1,s})},\ldots,y_r^{(q_{r,s})}].\] 
\end{definition}
\begin{lemma}\label{Consequenceuniversal}
Let $f$ and $h$ be polynomials of type $(e_1, e_2)$ and suppose that $V_{\mathcal{M}(f)}\preceq V_{\mathcal{M}(h)}$. Then, there exists $h'\in\langle f\rangle^{T_G}+\mathrm{Id}_G(\mathrm{UT}_3^{(-)},\Gamma(e_1,e_2))$  such that $\mathcal{M}(h)=\mathcal{M}(h')$.
\end{lemma}
\begin{proof}
From the proof of \Cref{oderseqfe1e2}, there exists a homomorphism $\psi:\mathcal{L}(X^G)\to\mathcal{L}(X^G)$ (as mentioned in the paragraph above \Cref{poltypee_1e_2}) and variables of trivial degree $g_{1}$, \dots, $g_k\in\mathcal{L}(X^G)$ such that
\[[\psi(\mathcal{M}(f)),g_{1},\ldots,g_{k}]=\mathcal{M}(h).
\]
Then,
\[h':=[\psi(f),g_{1},\ldots,g_{k}]\in \langle f \rangle^{T_G}+\mathrm{Id}_G(\mathrm{UT}_3^{(-)},\Gamma(e_1,e_2)).\] Applying \Cref{compativel}, we conclude that $\mathcal{M}(h)=\mathcal{M}(h')$.
\end{proof}
\begin{lemma}\label{finitoe1e2}
Let $J$ be a $T_G$-ideal such that $\mathrm{Id}_G(\mathrm{UT}_3^{(-)},\Gamma(e_1,e_2))\subseteq J$. Consider the set
\[\mathcal{A}^{(e_1,e_2)}=\{\textrm{polynomials of type $(e_1,e_2)$}\}\cap J.\]
Then, there exists a finite subset $\widetilde{A}\subseteq \mathcal{A}^{(e_1,e_2)}$ such that
\[\mathcal{A}^{(e_1,e_2)}\subseteq \langle \widetilde{A}\rangle^{T_G}+\mathrm{Id}_G(\mathrm{UT}_3^{(-)},\Gamma(e_1,e_2)).\]
\end{lemma}
\begin{proof}
Consider the sets
\[\mathcal{M}^{(e_1,e_2)}=\{\mathcal{M}(f)\mid f\in \mathcal{A}^{(e_1,e_2)}\},\]
\[\mathcal{V}^{(e_1,e_2)}= \{V_{\mathcal{M}(f)}\mid f\in \mathcal{M}^{(e_1,e_2)}\}.\]

By \Cref{Sequenceorder}, there exist $\mathcal{M}(f_1)$,\dots, $\mathcal{M}(f_t)\in \mathcal{M}^{(e_1,e_2)}$ such that $V_{\mathcal{M}(f_1)}$,\dots, $V_{\mathcal{M}(f_t)}$ are minimal elements of $\mathcal{V}^{(e_1,e_2)}$. By \Cref{oderseqfe1e2},
\[\mathcal{M}^{(e_1,e_2)}\subseteq \langle \mathcal{M}(f_1),\ldots, \mathcal{M}(f_t)\rangle^{T_G}+ \mathrm{Id}_G(\mathrm{UT}_3^{(-)},\Gamma(e_1,e_2)). \]
Now, we define $\widetilde{A}=\{f_1,\ldots, f_t\}$
and we shall show that
$\mathcal{A}^{(e_1,e_2)}\subseteq \langle \widetilde{A}\rangle^{T_G}+\mathrm{Id}_G(\mathrm{UT}_3^{(-)},\Gamma(e_1,e_2)).$
Suppose that there is a polynomial $f\in \mathcal{A}^{(e_1,e_2)}$  such that \[f\notin  \langle \widetilde{A}\rangle^{T_G}+\mathrm{Id}_G(\mathrm{UT}_3^{(-)},\Gamma(e_1,e_2))\] and 
$\deg_{y_i} f=d_i$ for $1\leq i\leq n$. Fixing the multidegree $(d_1,\ldots, d_n)$ in the variables $y$'s, consider the following sets
\[A^{(d_1,\ldots, d_n)}=\{h(x_{e_1},x_{e_2},y_1,\ldots, y_n)\in \mathcal{A}^{(e_1,e_2)}\mid h\notin\langle\widetilde{A}\rangle^{T_G}+\mathrm{Id}_G(\mathrm{UT}_3^{(-)},\Gamma(e_1,e_2))\, \},\]
\[\mathcal{M}^{(d_1,\ldots, d_n)}= \{\mathcal{M}(h)\mid h\in A^{(d_1,\ldots, d_n)} \}.\]
Note that $f\in A^{(d_1,\ldots, d_n)}$. Thus, without loss of generality we can suppose that
    \[\mathcal{M}(f)=\min_{<'}\, \mathcal{M}^{(d_1,\ldots, d_n)}.\]
Notice that there is $f_i\in \widetilde{A}$, for some $1\leq i\leq t$, such that $V_{\mathcal{M}(f_i)}\preceq V_{\mathcal{M}(f)}$. Then, by \Cref{Consequenceuniversal}, there is a polynomial $h$ of type $(e_1,e_2)$ such that 
\[h\in \langle f_i\rangle^{T_G}+ \mathrm{Id}_G(\mathrm{UT}_3^{(-)},\Gamma(e_1,e_2)) \]
and $\mathcal{M}(f)=\mathcal{M}(h)$. Without loss of generality, we shall suppose the coefficients of $\mathcal{M}(h)$ and $\mathcal{M}(f)$ in $h$ and $f$, respectively, to be equal to 1. 
If $\widehat{f}=h-f\neq 0$, then $\widehat{f} \in A^{(d_1,\ldots, d_n)}$ and $\mathcal{M}(\widehat{f}) <' \mathcal{M}(f)$, a contradiction. Therefore, $f=h$ and we have again a contradiction. Hence, we proved that \[\mathcal{A}^{(e_1,e_2)}\subseteq \langle \widetilde{A}\rangle^{T_G}+\mathrm{Id}_G(\mathrm{UT}_3^{(-)},\Gamma(e_1,e_2)).\]  
\end{proof}

\begin{theorem}\label{specthUniversal}
Let $G=\langle e_1,e_2\rangle$ be the free abelian group of rank $2$, $\mathbb{K}$ an infinite field, and $J$ be a $T_G$-ideal such that $\mathrm{Id}_G(\mathrm{UT}_3^{(-)}(\mathbb{K}),\Gamma(e_1,e_2))\subseteq J$. Then, there exists a finite subset $S\subseteq\mathcal{L}(X^G)$ such that $J=\langle S\rangle^{T_G}+\mathrm{Id}_G(\mathrm{UT}_3^{(-)},\Gamma(e_1,e_2))$.
\end{theorem}
\begin{proof}
Since $\mathbb{K}$ is an infinite field, $J$ is generated as $T_G$-ideal by its multihomogeneous polynomials. Consider the sets
\[\mathcal{B}^{g}=\{[x_g, y_{1}^{(k_1)},\ldots, y_{n}^{(k_n)}]\mid n\geq 1, k_i\geq 0\}\cap J,\]
\[\mathcal{A}^{(e_1,e_2)}=\{\textrm{polynomials of type $(e_1,e_2)$}\}\cap J,\]
where $g\in\{e_1,e_2,e_1e_2\}$. By \Cref{universalbaseidentidades}, we have 
\[J=\langle \mathcal{B}^{e_1}\cup\mathcal{B}^{e_2}\cup\mathcal{B}^{e_1e_2}\cup\mathcal{A}^{(e_1,e_2)}\rangle^{T_G}+ \mathrm{Id}_G(\mathrm{UT}_3^{(-)},\Gamma(e_1,e_2)).\]
From \Cref{finitecomutatoruniversal,finitoe1e2}, there exist finite subsets $\widetilde{B}^{g}\subseteq\mathcal{B}^{g}$  and $\widetilde{A}\subseteq \mathcal{A}^{(e_1,e_2)}$, such that
\[\mathcal{B}^{g}\subseteq \langle \widetilde{B}^{g}\rangle^{T_G}+\mathrm{Id}_G(\mathrm{UT}_3^{(-)},\Gamma(e_1,e_2)),\]
\[\mathcal{A}^{(e_1,e_2)}\subseteq \langle \widetilde{A}\rangle^{T_G}+\mathrm{Id}_G(\mathrm{UT}_3^{(-)},\Gamma(e_1,e_2)),\]
for each $g\in\{e_1,e_2,e_1e_2\}$. Hence, \[J= \langle \widetilde{B}^{e_1}\cup\widetilde{B}^{e_2}\cup\widetilde{B}^{e_1e_2}\cup\widetilde{A}\rangle^{T_G}+ \mathrm{Id}_G(\mathrm{UT}_3^{(-)},\Gamma(e_1,e_2)),\]
and the result is proved.
\end{proof}
\subsection{Almost Universal Grading}
Let $g\in G$ be an element of order $\geq 3$. We shall study $\mathrm{UT}_3^{(-)}$ endowed with the Almost Universal grading, that is, the elementary grading given by $\Gamma(g, g^{-1})$.

By direct computation, we can see that the following are graded identities for $(\mathrm{UT}_3^{(-)},\Gamma(g, g^{-1}))$:
\begin{equation}
    x_l=0,\quad  l\notin\mathrm{Supp}\,\Gamma(g,g^{-1}),
\end{equation}
\begin{equation}\label{almostuniversal2}
    [x_{1,l},x_{2,l}]=0\quad l\in\{g,g^{-1}\},
\end{equation}
\begin{equation}\label{almostuniversal3}
    [x_{1,l}, x_{1,h},x_{2,l}]=0\quad l,h\in\{g,g^{-1}\}, l\neq h,
\end{equation}
\begin{equation}\label{almostuniversal4}
    [y_1,y_2, x_{l}]=0 \quad  l\in\{g,g^{-1}\}.
\end{equation}
\begin{equation}\label{meatabelian}
    [[y_1,y_2],[y_3,y_4]]=0,
\end{equation}
In addition, if $\mathbb{K}$ is finite with $q$ elements, then the following polynomials
\begin{equation}\label{almostuniversal6}
[x_{l},y^{(q)}]-[x_{l},y]=0\quad  l\in\{g,g^{-1}\};
\end{equation}
\begin{equation}\label{almostuniversal7}
    [y_1, y_2, y_3^{(q)}]-[y_1, y_2, y_3]=0;
\end{equation}
\begin{equation}\label{almostuniversal8}
    [y_2, y_1, y_2^{(q-1)}]+ [y_2,y_1]-[y_2,y_1^{(q)}]-[y_2,y_1, y_2^{(q-1)}]
\end{equation}
are graded identities as well. Note that \eqref{meatabelian}, as well as \eqref{almostuniversal7} and \eqref{almostuniversal8} when $\mathbb{K}$ is finite, are (ordinary) polynomial identities of $\mathrm{UT}_2^{(-)}(\mathbb{K})$ (indeed, $\mathrm{UT}_2^{(-)}$ is a subalgebra of the identity component of $\Gamma(g,g^{-1})$). The characterization of $\mathrm{Id}_G(\Gamma(g,g^{-1}))$ is as follows.
\begin{theorem}\label{almostuniversalbaseidentidades}
Let $\mathbb{K}$ be an arbitrary field and consider the Almost Universal $G$-grading $\Gamma(g,g^{-1})$ on $\mathrm{UT}_3^{(-)}(\mathbb{K})$ (where $g^2\ne1$). Then, $\mathrm{Id}_G(\mathrm{UT}_3^{(-)},\Gamma(g,g^{-1}))$ is generated, as a $T_G$-ideal, by:
\renewcommand{\labelenumi}{(\roman{enumi})}
 \begin{enumerate}
 \item $x_l$, $l\notin\mathrm{Supp}\,\Gamma(g, g^{-1})$,
 \item $[x_{1,l},x_{2,l}]$, $l\in\{g,g^{-1}\}$,
 \item $  [x_{1,l}, x_{1,h},x_{2,l}]$, $l,h\in\{g,g^{-1}\}$, $l\neq h$,
 \item  $[y_1,y_2, x_{l}]$, $l\in\{g,g^{-1}\}$,
 \item $\mathrm{Id}(\mathrm{UT}_2^{(-)})\cap\mathcal{L}(Y)$.
\end{enumerate}
In addition, if $\mathbb{K}$ is finite containing $q$ elements, then we include:
\begin{enumerate}
\setcounter{enumi}{5}
\item $[x_{l},y^{(q)}]-[x_{l},y]$, $l\in\{g,g^{-1}\}$.
\end{enumerate}
A basis of the relatively free algebra $\mathcal{L}(X^G)$, modulo $\mathrm{Id}_G(\mathrm{UT}_3^{(-)},\Gamma(g,g^{-1}))$, constitute of all the following polynomials:
   \begin{align}
    &[x_l,y_1^{(p_1)},\ldots,y_r^{(p_r)}],\label{gensetaluniversalinfinito1}\\
    &[[x_{g},y_1^{(p_1)},\ldots,y_r^{(p_r)}],[x_{g^{-1}},y_1^{(q_1)},\ldots,y_r^{(q_r)}]]\label{gensetaluniversalinfinito2},
    \end{align}
 where $l\in \{g, g^{-1}\}$, $r\ge0$ and $0\le p_\ell,q_\ell<|\mathbb{K}|$ for each $\ell\in\{1,2,\ldots,r\}$, and a basis of the relatively free algebra of $\mathrm{UT}_2(\mathbb{K})^{(-)}$ in the variables $Y$.
 \end{theorem}
 
\begin{proof}
We shall work modulo the $T_G$-ideal of polynomial identities $T$, generated by the polynomial identities of the statement. First, consider a monomial 
    \[m=[x_{i_1,g_1},\ldots, x_{i_m, g_m}],\]
and let $s(m)=|\{j\mid g_j\neq 1\}|$. If $s(m)=0$, then $m\in\mathcal{L}(Y)$, so we are in the situation of a basis of the relatively free algebra of $\mathrm{UT}_2^{(-)}$.

Now, if $s(m)=1$, using the identity \eqref{almostuniversal4} and the Jacobi identity, we obtain
\[m=[x_{l}, y_{i_1},\ldots, y_{i_m}],\]
where  $l\in\{g,g^{-1}\}$ and $i_1\leq\cdots\leq i_k$. If $s(m)=2$, by repeated application of Jacobi identity
and using the identities \eqref{almostuniversal2} and \eqref{almostuniversal4}, we get
\[m=\sum\limits\left[[x_{g}, y_{i_1},\ldots, y_{i_m}],[x_{g^{-1}}, y_{j_1},\ldots, y_{j_s}]\right],\]
where $i_1\leq\cdots\leq i_m$ and $j_1\leq\cdots\leq j_m$. 

If $s(m)\geq 3$, it follows from identities \eqref{almostuniversal2} and \eqref{almostuniversal3} that $m\in T$. If the base field is finite, having $q$ elements, then identity \eqref{almostuniversal6} gives the upper bound $q_\ell,p_\ell<q$.

Now, let us prove that the given set of polynomials is linearly independent, modulo $\mathrm{Id}(\mathrm{UT}_3^{(-)},\Gamma(g,g^{-1}))$. Let
\[f=\sum\alpha_h h \in \mathrm{Id}_G (\mathrm{UT}_3^{(-)},\Gamma(g,g^{-1})),\]
where $\alpha_h\in\mathbb{K}$ and the sum runs over the set of polynomials of the statement. We may assume that the same set of variables appears in all the summands $h$. That is, without loss of generality, we can consider that either $f\in\mathcal{L}(Y)$, or
\[f=\sum\limits_{s}\alpha_s [x_l, y_{1}^{(p_{1,s})},\ldots, y_{r}^{(p_{r,s})}] \]
where $l\in \{g, g^{-1}\}$, $r \ge0$, $0\le p_{\ell, s}<|\mathbb{K}|$, for each $\ell\in\{1,2,\ldots,r\}$, or 
\[f=\sum\limits_{s=1}^t \alpha_{s}\left[[x_{g},y_1^{(p_{1,s})},\ldots,y_r^{(p_{r,s})}],[x_{g^{-1}},y_1^{(q_{1,s})},\ldots,y_r^{(q_{r,s})}]\right],\]
where $r\ge0$ and $0\leq p_{\ell, s},q_{\ell, s}<|\mathbb{K}|$, $1\leq \ell\leq r$.

The first case is known (see \cite{bahturin} when the base field is infinite, and \cite{GRivab} when the field is finite). For the second and third, we consider the evaluation $x_g\mapsto e_{12}$, $x_{g^{-1}}\mapsto e_{23}$ and $y_i\mapsto -\xi_ie_{11}+\xi_i'e_{33}$. Then, the evaluation gives:
$$
    [x_l,y_1^{(p_{1,s})},\ldots,y_r^{(p_{r,s})}]\mapsto\prod_{i=1}^r\xi_i^{p_{i,s}}e_{12},
$$
$$
    [[x_{g},y_1^{(p_{1,s})},\ldots,y_r^{(p_{r,s})}],[x_{g^{-1}},y_1^{(q_{1,s})},\ldots,y_r^{(q_{r,s})}]]\mapsto\left(\prod_{i=1}^r\xi_i^{p_{i,s}}\right)\left(\prod_{i=1}^r\xi_i^{\prime q_{i,s}}\right)e_{13}.
$$
Since the set of polynomials appearing on the right hand side is linearly independent, we obtain that the given set of polynomials is linearly independent modulo $\mathrm{Id}(\mathrm{UT}_3^{(-)},\Gamma(g,g^{-1}))$.
\end{proof}

The next step is to study the Specht property for the variety of graded Lie algebras generated by $\mathrm{UT}_3^{(-)}$ endowed with the Almost Universal grading when the base field $\mathbb{K}$ is infinite.

Using the notation of \Cref{universalgradingsection}, for each commutator $f=[x_l, y_{1}^{(k_1)},\ldots, y_{n}^{(k_n)}]$, where $l\in\{g, g^{-1}\}$, we associate the following finite sequence of non-negative integers
\begin{equation}
S_f= (k_1,\ldots, k_n).
\end{equation}
\begin{lemma}
 Let $(D(\mathbb{N}_0), \preceq)$ be the partially well-ordered set from \Cref{Sequenceorder}.  Consider the following commutators
\[f(x_1, y_1,\ldots, y_n)= [x_l, y_{1}^{(k_1)},\ldots, y_{n}^{(k_n)}],\]
\[h(x_1, y_1,\ldots, y_m)= [x_l, y_{1}^{(k'_1)},\ldots, y_{m}^{(k'_m)}],\]
where $l\in\{g,g^{-1}\}$ and $k_1$, \dots, $k_n$, $k'_1$, \dots, $k'_m\geq 0$. If $S_f\preceq S_h$, then $h\in \langle f\rangle^{T_G}+ \mathrm{Id}_G(\mathrm{UT}_3^{(-)},\Gamma(g,g^{-1}))$.
\end{lemma}
\begin{proof}
The proof is analogous to the proof of \Cref{oderSfSg}.
\end{proof}

\begin{lemma}\label{finitoalmostuniversalcommutadorsimples}
Let $l\in\{g, g^{-1}\}$ and $C$ be a subset of the set 
\[\{[x_l, y_1^{(k_1)},\ldots, y_n^{(k_n)}]\mid n\geq 0, k_1,\ldots, k_n\geq 0\}.\]  
Then, there exist $f_1,\ldots, f_s\in C$ such that
\[ \langle f_1,\ldots, f_s\rangle ^{T_G}+ \mathrm{Id}_G(\mathrm{UT}_3^{(-)},\Gamma(e_1,e_2))= \langle C\rangle^{T_G}+\mathrm{Id}_G(\mathrm{UT}_3^{(-)},\Gamma(g, g^{-1})).\]
\end{lemma}
\begin{proof}
The proof is similar to the proof of \Cref{finitecomutatoruniversal}.
\end{proof}
For each  polynomial of the kind $f=[[x_{g},y_1^{(k_1)}, \ldots, y_n^{(k_n)}], [x_{g^{-1}}, y_1^{(l_1)}, \ldots, y_n^{(l_n)}]]$, we associate the following finite sequence of elements in $\mathbb{N}_0^2$:
\[V_f= ( (k_1,l_1),\ldots,(k_n,l_n)).\]
\begin{lemma}\label{oderseqparesalmostuniversal}
 Let $(D(\mathbb{N}_0^2), \preceq)$ be the partially well-ordered set from \Cref{Sequenceorder}.  Consider the following commutators 
\[f(x_{g}, x_{g^{-1}}, y_1,\ldots, y_n)= [[x_{g}, y_{1}^{(k_1)},\ldots, y_{n}^{(k_n)}], [x_{g^{-1}}, y_1^{(l_1)},\ldots, y_n^{(l_n)}]],\]
\[h(x_{g}, x_{g^{-1}},y_1,\ldots, y_m)= [[x_{g}, y_{1}^{(k'_1)},\ldots, y_{m}^{(k'_m)}], [x_{g^{-1}}, y_1^{(l'_1)},\ldots, y_m^{(l'_m)}]],\]
where $0\leq k_1$,\dots, $k_n$, $l_1$, \dots, $l_n$, $k'_1$,\dots, $k'_m$, $l'_1$,\dots, $l'_m$. If $V_f\preceq V_h$, then  $h\in \langle f\rangle^{T_G}+ \mathrm{Id}_G(\mathrm{UT}_3^{(-)},\Gamma(g,g^{-1}))$
\end{lemma}
\begin{proof}
Since $V_f\preceq V_h$,  there exists a subsequence $((k'_{i_1}, l'_{i_1}), \ldots, (k'_{i_n},l'_{i_n}))$ of $V_h$ such that $(k_j,l_j) \le_2 (k'_{i_j},l'_{i_j})$ for every $ j \in \{1, \ldots,n\}$. Therefore $k_j \leq k'_{i_j}$ and $l_j\leq l'_{i_j}$  for every $j \in \{1,\ldots,n\}$. Then
    \[ \tilde{f}=f(x_{g}, x_{g^{-1}}, y_{i_1},\ldots, y_{i_n})=[[x_{g},y_{i_1}^{(k_1)}, \ldots, y_{i_n}^{(k_n)}], [x_{g^{-1}}, y_{i_1}^{(l_1)},\ldots, y_{i_n}^{(l_n)}]] \in \left< f \right>^{T_G}. \]
Evaluating as \[x_{g} \mapsto [x_{g},y_{i_1}^{(k'_{i_1}-k_1)},\ldots, y_{i_n}^{(k'_{i_n}-k_n)},y_{j_1}^{(k'_{j_1})}, \ldots, y_{j_{m-n}}^{(k'_{j_{m-n}})}],\]
\[x_{g^{-1}} \mapsto [x_{g^{-1}},y_{i_1}^{(l'_{i_1}-l_1)},\ldots, y_{i_n}^{(l'_{i_n}-l_n)},y_{j_1}^{(l'_{j_1})}, \ldots, y_{j_{m-n}}^{(l'_{j_{m-n}})}, ]\] 
on the polynomial $\tilde{f}$, where $\{j_{1}, \ldots, j_{m-n}\}=\{1,\ldots,m\}\setminus \{i_{1}, \ldots, i_{n}\}$, gives a new polynomial $\overline{f}\in \left< f \right>^{T_G}$. From the identity \eqref{almostuniversal4},
       \[ \overline{f}+ \mathrm{Id}_G(\mathrm{UT}_3^{(-)},\Gamma(g,g^{-1})) = h+ \mathrm{Id}_G(\mathrm{UT}_3^{(-)},\Gamma(g,g^{-1}))
       .\]
Therefore, $h\in \langle f\rangle^{T_G}+ \mathrm{Id}_G(\mathrm{UT}_3^{(-)},\Gamma(g,g^{-1}))$.  
\end{proof}

\begin{definition}\label{orderlinearmonomiosalmostuniversal}
Consider the following commutators of the same multidegree
\[f(x_{g},x_{g^{-1}}, y_1,\ldots, y_n)= [[x_{g}, y_{1}^{(k_1)},\ldots, y_{n}^{(k_n)}], [x_{g^{-1}}, y_1^{(l_1)},\ldots, y_n^{(l_n)}]],\]
\[h(x_{g}, x_{g^{-1}}, y_1,\ldots, y_n)= [[x_{g}, y_{1}^{(k'_1)},\ldots, y_{n}^{(k'_n)}], [x_{g^{-1}}, y_1^{(l'_1)},\ldots, y_n^{(l'_n)}]],\]
where $0\leq k_1,\ldots, k_n,l_1,\ldots, l_n,k'_1,\ldots, k'_n,l'_1$. 
We say that $f<_{au}' h$ if either
\begin{itemize}
\item[(a)] $\sum\limits_{i=1}^n k_i< \sum\limits_{i=1}^n k'_i$, or,
\item[(b)] $\sum\limits_{i=1}^n k_i= \sum\limits_{i=1}^n k'_i$ and $(k_1,\ldots, k_n,)<_{l} (k'_1,\ldots, k'_n,)$.
\end{itemize}
As before, $<_l$ is the lexicographic order on $\mathbb{N}_0^n$.
\end{definition}

\begin{definition}\label{monomiolideralmostuniversal}
We say that $f$ is a polynomial of type $(g, g^{-1})$ if $f$ is multihomogeneous and
\[f=\sum\limits_{s=1}^t \alpha_{s}[[x_{g},y_1^{(p_{1,s})},\ldots,y_r^{(p_{r,s})}],[x_{g^{-1}},y_1^{(q_{1,s})},\ldots,y_r^{(q_{r,s})}]],\] 
where $\alpha_s\in\mathbb{K}\setminus\{0\}$, for $1\leq s\leq t$. The leading monomial of $f$ with respect to $<_{au}'$ is
\[\nu(f)=\max_{<_{au}'}\{ [[x_{g},y_1^{(p_{1,s})},\ldots,y_r^{(p_{r,s})}],[x_{g^{-1}},y_1^{(q_{1,s})},\ldots,y_r^{(q_{r,s})}]]\mid 1\leq s\leq t\}.\]
\end{definition}

\begin{lemma}\label{Consequencealmostuniversal}
Let $f$, $h$ be polynomials of type $(g, g^{-1})$ and suppose that $V_{\nu(f)}\preceq V_{\nu (h)}$. Then, there exists $h'\in\langle f\rangle^{T_G}+\mathrm{Id}_G(\mathrm{UT}_3^{(-)},\Gamma(g,g^{-1}))$  such that $\nu (h)=\nu (h')$.
\end{lemma}
\begin{proof}
The proof is similar to the proof of \Cref{Consequenceuniversal}.
\end{proof}

\begin{lemma}\label{finitog1g-1}
Let $J$ be a $T_G$-ideal such that $\mathrm{Id}_G(\mathrm{UT}_3^{(-)},\Gamma(g,g^{-1}))\subseteq J$. Consider the set
\[\mathcal{A}^{(g,g^{-1})}=\{\textrm{polynomials of type $(g,g^{-1})$}\}\cap J.\]
Then, there exists a finite subset $\widetilde{A}\subseteq \mathcal{A}^{g,g^{-1}}$ such that
\[\mathcal{A}^{(g,g^{-1})}\subseteq \langle \widetilde{A}\rangle^{T_G}+\mathrm{Id}_G(\mathrm{UT}_3^{(-)},\Gamma(g,g^{-1})).\]
\end{lemma}
\begin{proof}
The proof is analogous to the proof of \Cref{finitoe1e2}.
\end{proof}
\begin{definition}\label{polynomialoftypetrivial}
A polynomial of trivial type is a multihomogeneous polynomial in $\mathcal{L}(Y)$.
\end{definition}
\begin{theorem}\label{specthalmostUniversal}
Let $G$ be an abelian group, $g\in G$ be such that $g^2\ne1$, $\mathbb{K}$ an infinite field and $J$ a $T_G$-ideal such that $\mathrm{Id}_G(\mathrm{UT}_3^{(-)}(\mathbb{K}),\Gamma(g,g^{-1}))\subseteq J$. Then, there exists a finite subset $S\subseteq\mathcal{L}(X^G)$ such that $J=\langle S\rangle^{T_G}+\mathrm{Id}_G(\mathrm{UT}_3^{(-)},\Gamma(g,g^{-1}))$.
\end{theorem}
\begin{proof}
Since $\mathbb{K}$ is an infinite field, $J$ is generated as $T_G$-ideal by its multihomogeneous polynomials. Thus, consider the sets
\[\mathcal{Y}=\{\textrm{polynomials of  trivial type }\}\cap J,\]
\[\mathcal{B}^{l}=\{[x_g, y_{1}^{(k_1)},\ldots, y_{n}^{(k_n)}]\mid n\geq 1, k_i\geq 0\}\cap J,\]
\[\mathcal{A}^{(g,g^{-1})}=\{\textrm{polynomials of type $(g,g^{-1})$}\}\cap J,\]
for each $l\in\{g, g^{-1}\}$. It follows from \Cref{almostuniversalbaseidentidades} that
\[J=\langle \mathcal{B}^{g}\cup\mathcal{B}^{g^{-1}}\cup\mathcal{Y}\cup\mathcal{A}^{(g,g^{-1})}\rangle^{T_G}+ \mathrm{Id}_G(\mathrm{UT}_3^{(-)},\Gamma(g,g^{-1})).\]
From \cite[Theorem 5.16]{bahturin} and \Cref{finitoalmostuniversalcommutadorsimples,finitog1g-1}, there exist finite subsets 
$\widetilde{Y}\subseteq \mathcal{Y}$, $\widetilde{B}^{l}\subseteq\mathcal{B}^{l}$  and $\widetilde{A}\subseteq \mathcal{A}^{(g,g^{-1})}$ such that
\[\mathcal{Y}\subseteq \langle \widetilde{Y}\rangle^{T_G}+\mathrm{Id}_G(\mathrm{UT}_3^{(-)},\Gamma(g,g^{-1})),\]
\[\mathcal{B}^{l}\subseteq \langle \widetilde{B}^{l}\rangle^{T_G}+\mathrm{Id}_G(\mathrm{UT}_3^{(-)},\Gamma(g,g^{-1})),\]
\[\mathcal{A}^{(g,g^{-1})}\subseteq \langle \widetilde{A}\rangle^{T_G}+\mathrm{Id}_G(\mathrm{UT}_3^{(-)},\Gamma(g,g^{-1})),\]
for each $l\in\{g,g^{-1}\}$. Hence, \[J= \langle \widetilde{Y}\cup \widetilde{B}^{g}\cup\widetilde{B}^{g^{-1}}\cup\widetilde{A}\rangle^{T_G}+ \mathrm{Id}_G(\mathrm{UT}_3^{(-)},\Gamma(g,g^{-1})).\]
The result is proved.
\end{proof}

\subsection{Almost Canonical Grading}
We denote $\mathbb{Z}_2=\langle 1\mid1+1=0 \rangle$ the cyclic group of order $2$ with additive notation. In this section, we rename the graded variables as $z_i=x_{1,i}$, for each $i\in\mathbb{N}$. We consider the almost canonical grading $\Gamma(1,1)$ on $\mathrm{UT}_3^{(-)}$.  It is straightforward to check that the following polynomials are $\mathbb{Z}_2$-graded identities for $(\mathrm{UT}_3^{(-)}, \Gamma(1,1))$:
\begin{equation}\label{almostcanonical1}
    [z_1, z_2, z_3]=0,
\end{equation}
\begin{equation}\label{almoscanonical2}
    [y_1,y_2, z_3]=0,
\end{equation}
\begin{equation}\label{meatabelianalmostcanonical}
    [[y_1,y_2],[y_3,y_4]]=0.
\end{equation}
In addition, if $\mathbb{K}$ is finite with $q$ elements, then the following polynomials are also graded polynomial identities:
\begin{equation}\label{almoscanonical3}
[z,y^{(q)}]-[z,y]=0,
\end{equation}
\begin{equation}\label{almostcanonical4}
    [y_1, y_2, y_3^{(q)}]-[y_1, y_2, y_3]=0,
\end{equation}
\begin{equation}\label{almoscanonical5}
    [y_2, y_1, y_2^{(q-1)}]+ [y_2,y_1]-[y_2,y_1^{(q)}]-[y_2,y_1, y_2^{(q-1)}].
\end{equation}
As before, the last 2 polynomial identities together with \eqref{meatabelianalmostcanonical} are polynomial identities for $\mathrm{UT}_2^{(-)}$. The characterization of $\mathrm{Id}_{\mathbb{Z}_2}(\mathrm{UT}_3^{(-)},\Gamma(1,1))$ is as follows.
\begin{theorem}\label{almoscanonicallbaseidentidades}
Let $\mathbb{K}$ be an arbitrary field and consider the Almost Canonical $\mathbb{Z}_2$-grading $\Gamma(1,1)$ on $\mathrm{UT}_3^{(-)}(\mathbb{K})$. Then, $\mathrm{Id}_{\mathbb{Z}_2}(\mathrm{UT}_3^{(-)},\Gamma(1,1))$ is generated by:
\renewcommand{\labelenumi}{(\roman{enumi})}
 \begin{enumerate}
\item  $[z_1, z_2, z_3]$,
\item  $[y_1,y_2, z]$,
 \item $\mathrm{Id}(\mathrm{UT}_2^{(-)})\cap\mathcal{L}(Y)$.
\end{enumerate}
In addition, if $\mathbb{K}$ is finite containing $q$ elements, then we include:
\begin{enumerate}
\setcounter{enumi}{3}
\item $[z,y^{(q)}]-[z,y]$.
\end{enumerate}
A basis of the relatively free algebra $\mathcal{L}(Z\cup Y)_{\mathbb{Z}_2}$, modulo $\mathrm{Id}_{\mathbb{Z}_2}(\mathrm{UT}_3^{(-)},\Gamma(1,1))$, is given by  a basis of the relatively free algebra of $\mathrm{UT}_2(\mathbb{K})^{(-)}$, in the variables $Y$,  and  the following polynomials:
   \begin{align}
    &[z,y_1^{(p_1)},\ldots,y_r^{(p_r)}],\\
    &[[z_i,y_1^{(p_1)},\ldots,y_r^{(p_r)}],[z_j,y_1^{(q_1)},\ldots,y_r^{(q_r)}]],
    \end{align}
 where $i\leq j$, $r\ge0$ and $0\le p_\ell,q_\ell<|\mathbb{K}|$ for each $\ell\in\{1,2,\ldots,r\}$. Moreover, If $i=j$, then 
 $(p_1,\ldots, p_r)<_{l} (q_1,\ldots, q_r)$, where $<_l$ is the lexicographic order on $\mathbb{N}_0^r$.
 \end{theorem}
 \begin{proof}
  We shall work modulo the $T_{\mathbb{Z}_2}$-ideal of polynomial identities $T$, generated by the polynomial identities of the statement. First, consider a monomial 
    \[m=[x_{i_1,g_1},\ldots, x_{i_m, g_m}],\]
and let $s(m)=|\{j\mid g_j\neq 0\}|$. If $s(m)=0$, then $m\in\mathcal{L}(Y)$, so we are in the situation of a basis of the relatively free algebra of $\mathrm{UT}_2^{(-)}$.

If $s(m)=1$, then, using identity \eqref{almoscanonical2}, we may assume that
\[m=[z, y_{i_1},\ldots, y_{i_m}],\]
where $i_1\leq\cdots\leq i_k$. If $s(m)=2$, then using identity \eqref{almoscanonical2} and by repeated application of Jacobi identity, we get
\[m=\sum\limits\left[[z_i, y_{i_1},\ldots, y_{i_m}],[z_j, y_{j_1},\ldots, y_{j_s}]\right],\]
where $i\leq j$, $i_1\leq\cdots\leq i_m$ and $j_1\leq\cdots\leq j_s$ in each summand. Thus, $m$ can be written as
\[m=\sum[[z_i,y_1^{(p_1)},\ldots,y_r^{(p_r)}],[z_j,y_1^{(q_1)},\ldots,y_r^{(q_r)}]],\]
where $q_\ell,p_\ell\geq 0$. In addition, if the base field is finite and having $q$ elements, then identity \eqref{almoscanonical3} gives the upper bound $q_\ell,p_\ell<q$. If $i=j$, by anticommutativity of the commutator, we can suppose that $(p_1,\ldots, p_r)<_{l} (q_1,\ldots, q_r)$.

If $s(m)\geq 3$, it follows from identity \eqref{almostcanonical1} that $m\in T$.

Now, let us prove that the given set of polynomials is linearly independent, modulo $\mathrm{Id}_{\mathbb{Z}_2}(\mathrm{UT}_3^{(-)},\Gamma(1,1))$. Let
\[f=\sum\alpha_h h \in \mathrm{Id}_{\mathbb{Z}_2} (\mathrm{UT}_3^{(-)},\Gamma(1,1)),\]
where $\alpha_h\in\mathbb{K}$ and the sum runs over the set of polynomials of the statement. We may assume that the same set of variables appears in all the summands of $h$. That is, without loss of generality, we can consider that either $f\in\mathcal{L}(Y)$, or
\[f=\sum\limits_{s}\alpha_s [z, y_{1}^{(p_{1,s})},\ldots, y_{r}^{(p_{r,s})}],\]
where $l\in \{g, g^{-1}\}$, $r \ge0$, $0\le p_{\ell, s}<|\mathbb{K}|$, for each $\ell\in\{1,2,\ldots,r\}$, or 
\[f=\sum\limits_{s=1}^t \alpha_{s}\left[[z_{i},y_1^{(p_{1,s})},\ldots,y_r^{(p_{r,s})}],[z_{j},y_1^{(q_{1,s})},\ldots,y_r^{(q_{r,s})}]\right],\]
where $i\leq j$, $r\ge0$ and $0\leq p_{\ell, s},q_{\ell, s}<|\mathbb{K}|$, $1\leq \ell\leq r$. In addition, if $i=j$, then $(p_{1,s},\ldots, p_{r,s})<_{l} (q_{1,s},\ldots, q_{r,s})$.

As before, the first case is known. For the second case, we consider the evaluation $z\mapsto e_{12}$ and $y_i\mapsto \xi_i e_{22}$. Then, the evaluation gives
$$
    [x_l,y_1^{(p_{1,s})},\ldots,y_r^{(p_{r,s})}]\mapsto\prod_{i=1}^r\xi_i^{p_{i,s}}e_{12}.
$$
Since the set of polynomial on the right side of the equation is linearly independent, it follows that $\alpha_s=0$, for all $s$.

For the third case, if $i<j$, we consider the evaluation $x_i\mapsto e_{12}$, $x_{j} \mapsto e_{23}$ and $y_i\mapsto -\xi_ie_{11}+\xi_i'e_{33}$. Then, the evaluation gives
\[
    [[x_{i},y_1^{(p_{1,s})},\ldots,y_r^{(p_{r,s})}],[x_{j},y_1^{(q_{1,s})},\ldots,y_r^{(q_{r,s})}]]\mapsto\left(\prod_{i=1}^r\xi_i^{p_{i,s}}\right)\left(\prod_{i=1}^r\xi_i^{\prime q_{i,s}}\right)e_{13}.\]
    Since the set of polynomials appearing on the right hand side is linearly independent, we obtain that the given set of polynomials is linearly independent modulo $\mathrm{Id}(\mathrm{UT}_3^{(-)},\Gamma(1,1))$.

Now, assume that $i=j$ and suppose that 
    \[f=\sum\limits_{s=1}^t \alpha_{s}\left[[z_{i},y_1^{(p_{1,s})},\ldots,y_r^{(p_{r,s})}],[z_{i},y_1^{(q_{1,s})},\ldots,y_r^{(q_{r,s})}]\right]\in  \mathrm{Id}_{\mathbb{Z}_2}(\mathrm{UT}_3^{(-)},\Gamma(1,1)).\]
    Linearizing the variable $z_{i}$ in the polynomial $f$, we obtain a polynomial $g$ such that
    \begin{equation*}
    \begin{split}
    g= & \sum\limits_{s=1}^t \alpha_{s}\left[[z_{i},y_1^{(p_{1,s})},\ldots,y_r^{(p_{r,s})}],[z_{j},y_1^{(q_{1,s})},\ldots,y_r^{(q_{r,s})}]\right]+ \\
    & \sum\limits_{s=1}^t \alpha_{s}\left[[z_{j},y_1^{(p_{1,s})},\ldots,y_r^{(p_{r,s})}],[z_{i},y_1^{(q_{1,s})},\ldots,y_r^{(q_{r,s})}]\right]\in  \mathrm{Id}_{\mathbb{Z}_2}(\mathrm{UT}_3^{(-)},\Gamma(1,1))
\end{split}
\end{equation*}
Suppose that $i<j$. It follows that 
  \begin{equation*}
    \begin{split}
    g= & \sum\limits_{s=1}^t \alpha_{s}\left[[z_{i},y_1^{(p_{1,s})},\ldots,y_r^{(p_{r,s})}],[z_{j},y_1^{(q_{1,s})},\ldots,y_r^{(q_{r,s})}]\right]- \\
    & \sum\limits_{s=1}^t \alpha_{s}\left[[z_{i},y_1^{(q_{1,s})},\ldots,y_r^{(q_{r,s})}], [z_{j},y_1^{(p_{1,s})},\ldots,y_r^{(p_{r,s})}]\right]\in  \mathrm{Id}_{\mathbb{Z}_2}(\mathrm{UT}_3^{(-)},\Gamma(1,1)).
\end{split}
\end{equation*}
Recall that  $(p_{1,s},\ldots, p_{r,s})<_{l} (q_{1,s},\ldots, q_{r,s})$. By the previous case, the given set of polynomials is linearly independent modulo $\mathrm{Id}_{\mathbb{Z}_2}(\mathrm{UT}_3^{(-)},\Gamma(1,1))$. Thus, each $\alpha_s=0$.
 \end{proof}
 
Next, we study the Specht property for the variety of graded Lie algebras generated by $\mathrm{UT}_3^{(-)}$ endowed with the Almost Canonical grading when the base field $\mathbb{K}$ is infinite.
\begin{definition}
A pure polynomial of type $(1,1)$ is a nonzero multihomogeneous polynomial $f$ of the kind
\[f=\sum\limits_{s=1}^t \alpha_{s}[[z_1,y_1^{(p_{1,s})},\ldots,y_r^{(p_{r,s})}],[z_2,y_1^{(q_{1,s})},\ldots,y_r^{(q_{r,s})}]].\] 
A polynomial of type $(1,1)$ is a nonzero multihomogeneous polynomial $g$ of the kind
\[g=\sum\limits_{s=1}^t \alpha_{s}[[z,y_1^{(p_{1,s})},\ldots,y_r^{(p_{r,s})}],[z,y_1^{(q_{1,s})},\ldots,y_r^{(q_{r,s})}]].\] 
\end{definition}
Thus, a pure polynomial of type $(1,1)$ is linear on the variables of non-trivial degree, while a polynomial of type $(1,1)$ is never linear.

 \begin{theorem}\label{specthalmostcanonical}
 Let $\mathbb{K}$ be an infinite field of characteristic different from 2 and 
let $J$ be a $T_{\mathbb{Z}_2}$-ideal such that $\mathrm{Id}_{\mathbb{Z}_2}(\mathrm{UT}_3^{(-)}(\mathbb{K}),\Gamma(1,1))\subseteq J$. Then $J$ is finitely generated as $T_{\mathbb{Z}_2}$-ideal.
\end{theorem}
\begin{proof}
Since $\mathbb{K}$ is an infinite field, $J$ is generated, as $T_{\mathbb{Z}_2}$-ideal, by its multihomogeneous polynomials. By \Cref{almoscanonicallbaseidentidades},
\[J=\langle \mathcal{Y}\cup\mathcal{B}^{1}\cup\mathcal{A}^{(1,1)}\cup \mathcal{B}^{(1,1)}\rangle^{T_{\mathbb{Z}_2}}+ \mathrm{Id}_{\mathbb{Z}_2}(\mathrm{UT}_3^{(-)},\Gamma(1,1)),\]
where
\[\mathcal{Y}=\{\textrm{polynomials of trivial type}\}\cap J,\]
\[\mathcal{B}^{1}=\{[z, y_{1}^{(k_1)},\ldots, y_{n}^{(k_n)}]\mid n\geq 1, k_i\geq 0\}\cap J,\]
\[\mathcal{A}^{(1,1)}=\{\textrm{pure polynomials of type $(1,1)$}\}\cap J,\]
\[\mathcal{B}^{(1,1)}=\{\textrm{polynomials of type $(1,1)$}\}\cap J.\]
Since $\ch\, \mathbb{K}\neq 2$, the multilinearization process in the variable $z$ gives $\mathcal{B}^{(1,1)}\subseteq \langle \mathcal{A}^{(1,1)}\rangle^{T_{\mathbb{Z}_2}}$. Therefore, 
\[J=\langle \mathcal{Y}\cup\mathcal{B}^{1}\cup\mathcal{A}^{(1,1)}\rangle^{T_{\mathbb{Z}_2}}+ \mathrm{Id}_{\mathbb{Z}_2}(\mathrm{UT}_3^{(-)},\Gamma(1,1)).\]
By \cite[Theorem 5.16]{bahturin} and employing similar arguments as those used in  \Cref{finitoalmostuniversalcommutadorsimples,finitog1g-1}, there exist finite subsets 
$\widetilde{Y}\subseteq \mathcal{Y}$, $\widetilde{B}^{1}\subseteq\mathcal{B}^{1}$  and $\widetilde{A}\subseteq \mathcal{A}^{(1,1)}$, such that
\[\mathcal{Y}\subseteq \langle \widetilde{Y}\rangle^{T_{\mathbb{Z}_2}}+\mathrm{Id}_{\mathbb{Z}_2}(\mathrm{UT}_3^{(-)},\Gamma(1,1)),\]
\[\mathcal{B}^{1}\subseteq \langle \widetilde{B}^{1}\rangle^{T_{\mathbb{Z}_2}}+\mathrm{Id}_{\mathbb{Z}_2}(\mathrm{UT}_3^{(-)},\Gamma(1,1)),\]
\[\mathcal{A}^{(1,1)}\subseteq \langle \widetilde{A}\rangle^{T_{\mathbb{Z}_2}}+\mathrm{Id}_{\mathbb{Z}_2}(\mathrm{UT}_3^{(-)},\Gamma(1,1)).\]
It follows that \[J= \langle \widetilde{Y}\cup \widetilde{B}^{1}\cup\widetilde{A}\rangle^{T_{\mathbb{Z}_2}}+ \mathrm{Id}_{\mathbb{Z}_2}(\mathrm{UT}_3^{(-)},\Gamma(1,1)).\]
By \Cref{almoscanonicallbaseidentidades}, $\mathrm{Id}_{\mathbb{Z}_2}(\mathrm{UT}_3^{(-)},\Gamma(1,1))$ has a finite basis as $T_{\mathbb{Z}_2}$-ideal. Thus, we can conclude that $J$ is finitely generated as $T_{\mathbb{Z}_2}$-ideal.
\end{proof}

The condition $\mathrm{char}\,\mathbb{K}\ne2$ is essential in the previous theorem, as illustrated by the following.
\begin{proposition}\label{almostcanonical_nonspecht}
Let $\mathbb{K}$ be an infinite field of characteristic 2. Then, the ideal $\mathrm{Id}_{\mathbb{Z}_2}(\mathrm{UT}_3^{(-)}(\mathbb{K}),\Gamma(1,1))$ does not satisfy the Specht property.
\end{proposition}
\begin{proof}
Let $\Delta$ be the canonical $\mathbb{Z}$-grading on $\mathrm{UT}_3^{(-)}$, and let $\alpha:\mathbb{Z}\to\mathbb{Z}_2$ be the canonical projection. Then, $\Delta^\alpha\cong\Gamma(1,1)$. Let $\pi:\mathbb{K}\langle X^{\mathbb{Z}_2}\rangle\to\mathbb{K}\langle X^\mathbb{Z}\rangle$ be defined by $\pi(y_i)=x_{i,0}+x_{i,2}$ and $\pi(z_i)=x_{i,1}$. For each $k>1$, let
$$
c_k:=[z_1,y_1,\ldots,y_k,z_1].
$$
Then, an easy computation shows that $\pi(c_k)+\mathrm{Id}(\Delta)=[x_{1,1},x_{0,1},\ldots,x_{0,k},x_{1,1}]+\mathrm{Id}(\Delta)$. From the proof of \cite[Theorem 20]{correa}, one has
$$
\langle\pi(c_2)\rangle_{T_\mathbb{Z}}+\mathrm{Id}(\Delta)\subsetneq\langle\pi(c_2),\pi(c_3)\rangle_{T_\mathbb{Z}}+\mathrm{Id}(\Delta)\subsetneq\cdots
$$
Hence, it follows from \Cref{transferSpecht} that $\Gamma(1,1)$ does not satisfy the Specht property.
\end{proof}

\subsection{Remaining Grading}
We consider the Remaining grading on $\mathrm{UT}_3^{(-)}$. Let $g\in G$ be any non-trivial element and consider the elementary grading given by $\Gamma(g,1)$.  We denote $z_i=x_{i,g}$ and $y_i=x_{i,1}$ for all $i\in\mathbb{N}$. It is straightforward to verify that the following polynomials are graded polynomial identities for $(\mathrm{UT}_3^{(-)},\Gamma(g,1))$:
\begin{equation}\label{remaining3}
   x_{1,l}=0,\quad l\notin \{1,g\},
\end{equation}
\begin{equation}\label{remaining1}
    [z_1, z_2],
\end{equation}
\begin{equation}\label{remaining2}
    [z_1, [y_2,y_3], [y_4, y_5]],
\end{equation}
\begin{equation}\label{meatabelianaremaining}
    [[y_1,y_2],[y_3,y_4]].
\end{equation}
In addition, if $\mathbb{K}$ is finite with $q$ elements, then the following elements

\begin{equation}\label{remaining6}
    [z,y_1^{(q)},y_2^{(q)}]-  [z,y_1^{(q)},y_2]-  [z,y_1,y_2^{(q)}]+ [z,y_1,y_2],
\end{equation}
\begin{equation}
    [z,[y_1,y_2],y_3^{(q)}]- [z,[y_1,y_2], y_3],
\end{equation}
\begin{equation}\label{remaining4}
    [y_1, y_2, y_3^{(q)}]-[y_1, y_2, y_3],
\end{equation}
\begin{equation}\label{remaining5}
    [y_2, y_1, y_2^{(q-1)}]+ [y_2,y_1]-[y_2,y_1^{(q)}]-[y_2,y_1, y_2^{(q-1)}],
\end{equation}
are graded identities for $(\mathrm{UT}_3^{(-)},\Gamma(g,1))$. As before, the polynomials \eqref{meatabelianaremaining}, \eqref{remaining4} and \eqref{remaining5} are ordinary polynomial identities of $\mathrm{UT}_2^{(-)}$.
\begin{definition}\label{comutadorestipo1}
A commutator of type $1$ is
\[[z, y_{i_1}^{(k_1)},\ldots, y_{i_n}^{(k_n)}],\]
where $i_1<\cdots <i_n$ and $0\leq k_i< 2|\mathbb{K}|$ for each $1\leq i\leq n$.
\end{definition}

\begin{definition}\label{comutadorestipo2}
Let $\mathcal{B}$ be a basis of the relatively free algebra of $\mathrm{UT}_2(\mathbb{K})^{(-)}$ in the variables $Y$. A commutator of type 2 is a commutator of the kind 
\[[m, z, y_{i_1}^{(q_1)},\ldots, y_{i_n}^{(q_n)}],\]
where $m\in\mathcal{B}$, $|m|\geq 2$ and $0\leq q_i< |\mathbb{K}|$ for all $1\leq i\leq n$.
\end{definition}

\begin{lemma}\label{conjuntogeradorremaining}
Let $\mathbb{K}$ be an arbitrary field and let $T$ be the $T_{G}$-ideal generated by the identities listed above. Then, the vector space $L\langle X^G\rangle$, modulo $T$, is generated by a basis $\mathcal{B}$ of the relatively free algebra of $\mathrm{UT}_2(\mathbb{K})^{(-)}$ in the variables $Y$, and all commutators of type 1 and 2.
\end{lemma}
\begin{proof}
 We shall work modulo $T$. First, consider a monomial 
    \[m=[x_{i_1,g_1},\ldots, x_{i_m, g_m}],\]
and let $s(m)=|\{j\mid g_j\neq 1\}|$. If $s(m)=0$, then $m\in\mathcal{L}(Y)$. This means that $m$ can be written as linear combination of elements of $\mathcal{B}$.

If $s(m)\geq 2$, by identity \eqref{remaining1}, we have $m=0$.

If $s(m)=1$, it is enough to prove the lemma for
\[m=[z, y_{i_1},\ldots, y_{i_n}].\]
Assume that $i_1\leq i_2\leq \cdots \leq i_s$ for some $1\leq s<n$. Then,
\[m=[z, y_{i_1},\ldots, y_{i_{s+1}}, y_{i_s},\ldots, y_{i_n}]-[ y_{i_s}, y_{i_{s+1}},[z, y_{i_1},\ldots, y_{i_{s-1}}],\ldots, y_{i_n}].\]
Using the identity \eqref{remaining2} and the Jacoby identity, we have that the second part of $m$ is a linear combination of polynomials of type $2$. If $i_{s-1}\leq i_{s+1}$, then the result follows by induction on $n-s$. Otherwise, we repeat this process for $y_{i_{s-1}}$ and $y_{i_{s+1}}$ on the first part of $m$. The process ends in at most $s$ steps, so we obtain a polynomial
\[[z, y_{j_1},\ldots, y_{j_{s+1}}, y_{i_{s+2}},\ldots, y_{i_n}],\]
where $j_1\leq\cdots \leq j_{s+1}$. The result follows by induction.
\end{proof}

\begin{remark}\label{evaluacionesremark}
Note that $A=span_{\mathbb{K}}\{e_{22},e_{33}, e_{23}\}\cong \mathrm{UT}_2^{(-)}$ as $\mathbb{K}$-algebra.  Consider an evaluation $\psi_0:\mathcal{L}(Y)\to A$ given by   \[\psi_0(y_i)= \xi_i' e_{22} +\xi''_i e_{23},\]where $\xi'_i, \xi''_i\in\mathbb{K}$ for all $i\geq 1$. We define $\psi: \mathcal{L}(X^G)\to \mathrm{UT}_3^{(-)}$ via 
\[ \psi(y_i)=  \xi_i e_{11}+ \psi_0(y_i),\quad \psi(z)=e_{12},\quad \psi(w)=0,\]
where $\xi_i\in\mathbb{K}$ for all $i\geq 1$ and $w\notin Y\cup\{z\}$.

Let $\mathcal{B}$ be a basis of the relatively free algebra of $\mathrm{UT}_2(\mathbb{K})^{(-)}$, in the variables $Y$, and consider the linear map $\pi: span_{\mathbb{K}}\{e_{23}\}\to \mathbb{K}$ given by $\pi(\lambda e_{23})=\lambda$. If $m\in\mathcal{B}$ and $|m|\geq 2$, we have that
\begin{equation*}
\begin{split}
\psi([m, z, y_{i_1}^{(q_1)},\ldots, y_{i_n}^{(q_n)}])= &[\psi_0(m), e_{12}]\prod\limits_{i=1}^n \xi_i^{q_i}\\
=& -(\pi(\psi_0(m))\left(\prod\limits_{i=1}^n \xi_i^{q_i}\right)e_{13}.
\end{split}
\end{equation*}
Moreover, note that the image of an evaluation of $m$ on $A$ spans a subspace of dimension at most $1$. As a consequence, if $f:A\to\mathbb{K}$ is a linear map, then there is $\lambda\in\mathbb{K}$ such that 
\[f\circ \psi_0 (m)=\lambda \pi\circ \psi_0(m),\]
for all $m\in\mathcal{B}$ such that $|m|\geq 2$.
\end{remark}

\begin{lemma}\label{liconjuntopoltipo2}
Let $\mathbb{K}$ be an arbitrary field. The set of all the commutators of type $2$ is linearly independent, modulo $\mathrm{Id}_G(\mathrm{UT}_3^{(-)},\Gamma(g,1))$.
\end{lemma}
\begin{proof}
Let $S_0$ be a finite set of commutators of type $2$, and assume that every polynomial in $S_0$ contains the same set of variables. Write
\[f=\sum_{h\in S_0}\alpha_h h \in \mathrm{Id}_{G} (\mathrm{UT}_3^{(-)},\Gamma(g,1)).  \]
Given $m\in\mathcal{B}$, we shall denote
\[
[m,z,Y^{(q_{1},\ldots,{q_{n}})}]:=[h, z, y_1^{(q_{1})},\ldots, y_n^{(q_{n})}].\] 
For $h\in S_0$, we shall write $h=[m_h,z,Y^{(q_{1,h},\ldots,q_{n,h})}]$. Given $Q:={(q_1,\ldots,q_n)}$, define
\[B_Q=\{m\in\mathcal{B} \mid[m, z, Y^{(q_1,\ldots,q_n)}]\in S_{0}\}.\] 

Fix $h_0\in S_0$ and let $Q=(q_{1,h_0},\ldots,q_{n,h_0})$ and $m_{0}= m_{h_0}$. Consider the algebra $A=span_{\mathbb{K}}\{e_{22},e_{33},e_{23}\}$. By \Cref{inverseprocess}, there exist a finite set $\Lambda$, a family of homomorphism $\{\psi_{0j}:\mathcal{L}(Y)\to A\}_{j\in\Lambda}$ and a family of linear maps $\{f_{j}:A\to\mathbb{K}\}_{j\in\Lambda}$  
such that
\[\sum_{j\in\Lambda}f_{j}\circ\psi_{0j}\left(\sum\limits_{m\in B_Q} \alpha_m m\right)=\alpha_{m_0}. \]
By \Cref{evaluacionesUT}, for each $j\in\Lambda$, there exists a homomorphism $\psi'_{0j}:\mathcal{L}(Y)\to A$ such that 
\[\psi_{0j}(m)=\psi'_{0j}(m),\] 
for each commutator $m$ of length at least $2$, and, for each $y\in Y$, the entry $(3,3)$ of $\psi'_{0j}(y)$ is $0$.

From \Cref{evaluacionesremark}, there is a family of $G$-graded homomorphism $\{\psi_{j}:\mathcal{L}(X^G)\to \mathrm{UT}_3^{(-)}\}_{j\in\Lambda}$ and a family of scalars $\{\lambda_j\mid j\in\Lambda\}$ such that
\[f_j\circ \psi'_{0j}(m)=\lambda_j \pi\circ\psi'_{0j}(m),\]
\[\psi_j([m,z, y_1^{(k_1)},\ldots, y_n^{(k_n)}])=-\lambda_j\pi(\psi'_{0j}(m)) \left(\prod\limits_{i=1}^n \xi_i^{k_i}\right) e_{13},\]
for each commutator $m$ of length at least $2$. Then, we have
\begin{eqnarray*}
0 & =&\sum\limits_{j\in\Lambda} \lambda_j \psi_j(f) = \sum\limits_{j\in\Lambda}\lambda_j\psi_j\left( \sum\limits_{m_h\in B_Q}\alpha_h h + \sum\limits_{m_h\notin B_Q} \alpha_h h \right) \\
&=&\left(-\left( \sum\limits_{j\in\Lambda} \lambda_j \pi\circ\psi'_{0j}(\sum\limits_{m_h\in B_Q}\alpha_h m_h)\right)\prod\limits_{i=1}^n \xi_i^{q_i, h_0}- \sum\limits_{j\in \Lambda}\sum\limits_{m_h\notin B_Q} \left(\lambda_j \pi\circ\psi'_{0j}(\alpha_h m_h) \prod\limits_{i=1}^n \xi_i^{q_{i,h}}\right)\right) e_{13} \\
& = &\left( -\alpha_{h_0} \prod\limits_{i=1}^n \xi_i^{q_i,h_0} - \sum\limits_{j\in\Lambda}\sum\limits_{m_h\notin B_Q}\left(f_i\circ\psi'_{0_j}(\alpha_h m_h) \prod\limits_{i=1}^n \xi_i^{q_{i,h}}\right) \right) e_{13} .
\end{eqnarray*}
Hence
\[0=-\alpha_{h_0} \prod\limits_{i=1}^n \xi_i^{q_i, h_0} - \sum\limits_{j\in\Lambda}\sum\limits_{m_h\notin B_Q}\left(f_j\circ\psi'_{0_j}(\alpha_h m_h)\prod\limits_{i=1}^n \xi_i^{q_{i,h}}\right).\]
It follows that $\alpha_{h_0}=0$. Therefore, $S_0$ is linearly independent modulo $\mathrm{Id}_{G} (\mathrm{UT}_3^{(-)},\Gamma(g,1))$.
\end{proof}
\begin{theorem}\label{baseremaininggrading}
Let $\mathbb{K}$ be an infinite field, $G$ any group and consider the Remaining $G$-grading $\Gamma(g,1)$ on $\mathrm{UT}_3^{(-)}(\mathbb{K})$ (where $g\ne1$). Then, $\mathrm{Id}_G(\mathrm{UT}_3^{(-)},\Gamma(g,1))$ is generated, as a $T_G$-ideal, by:
\renewcommand{\labelenumi}{(\roman{enumi})}
 \begin{enumerate}
 \item $x_{1,l}=0$, $l\notin \{1,g\}$,
 \item $[z_1, z_2]=0$,
 \item $[[y_1,y_2],[y_3,y_4]]=0$,
 \item $[z_1, [y_2,y_3], [y_4, y_5]]=0$.
 \end{enumerate}
A basis of the relatively free algebra $\mathcal{L}(X^G)$, modulo $\mathrm{Id}_G(\mathrm{UT}_3^{(-)},\Gamma(g,1))$, is given by  a basis $\mathcal{B}$ of the relatively free algebra of $\mathrm{UT}_2(\mathbb{K})^{(-)}$, in the variables $Y$, and all commutators of type 1 and 2 (see \Cref{comutadorestipo1,comutadorestipo2}).
\end{theorem}
\begin{proof}
By \Cref{conjuntogeradorremaining}, it is sufficient to prove that the generating  set 
$S$ formed by the elements of $\mathcal{B}$ and commutators of type 1 and 2 is linearly independent, modulo $\mathrm{Id}_G(\mathrm{UT}_3^{(-)},\Gamma(g,1))$. Let $S_0$ be a finite subset of $S$ and assume that every polynomial in $S_0$ contains the same set of variables. Let
\[f=\sum_{h\in S_0}\alpha_h h \in \mathrm{Id}_{G} (\mathrm{UT}_3^{(-)},\Gamma(g,1)),\]
where $\alpha_h\in\mathbb{K}$. Either $f\in\mathcal{L}(Y)$ or $f$ is a linear combination of commutators of type $1$ and $2$. The result is already known if $f\in\mathcal{L}(Y)$. So, assume that $f$ is a linear combination of polynomials of type $1$ and $2$, and let $z$ be the variable of non-trivial degree that appears in the polynomials of $S_0$.

Consider the evaluation $z\mapsto  e_{12}$ and $y_i\mapsto \xi_i e_{22}+ \xi'_i e_{33}$, for $i\in\mathbb{N}$. Then,
\[[z, y_1^{(k_1)},\ldots y_n^{(k_n)}]\mapsto \left(\prod\limits_{i=1}^n \xi_i^{k_i}\right) e_{12},\]
\[[m, z, y_{1}^{(q_1)},\ldots, y_{n}^{(q_n)}]\mapsto 0.\]

Since the above polynomials are linearly independent, we get that $f$ may be assumed as a linear combination of commutators of type 2. Thus, the result follows from \Cref{liconjuntopoltipo2}.
\end{proof}

\noindent\textbf{Problem.} Find a basis of $\mathrm{Id}_G(\mathrm{UT}_3^{(-)}(\mathbb{K}),\Gamma(g,1))$ when $\mathbb{K}$ is a finite field. Find a basis of the respective relatively free algebra.

Finally, we shall study the Specht property for the variety of graded Lie algebras generated by $\mathrm{UT}_3^{(-)}$ endowed with the Remaining grading when the base field $\mathbb{K}$ is infinite.
\begin{definition}
A basic polynomial is a multihomogeneous polynomial that is a linear combination of commutators of type $1$ and $2$. A polynomial of type 2 is a multihomogeneous polynomial that is a linear combination of commutators of type $2$. 
\end{definition}

We define the following sets
\[\mathcal{H}^{(1,2)}= \{f\mid\textrm{$f$ is a basic polynomial, $\deg f\geq 1$} \},\]
\[\mathcal{H}^{(2)}= \{ f\mid\textrm{$f$ is a polynomial of type $2$, $\deg f\geq 1$}\},\]
\[\mathcal{H}= \mathcal{H}^{(1,2)}\setminus \mathcal{H}^{(2)}.\]
\begin{remark}\label{polynomiosH}
If $f\in\mathcal{H}$, then $f$ is a multihomogeneous polynomial in the variables $z$, $y_1$, \dots, $y_n$, for some $n\geq 1$, and
\[f= \alpha c +g,\]
where $\alpha\in\mathbb{K}\setminus\{0\}$, $c$ is a commutator of type 1 and $g\in H^{(2)}$. We do not exclude the case where $g$ could be zero. If $g$ is non-zero, then $\deg_{y_i} c=\deg_{y_i} g$, for all $i\in\{1,\ldots, n\}$.
\end{remark}
\begin{remark}\label{basisUT2infinite}
A basis $\mathcal{B}$ of the relatively free algebra of $\mathrm{UT}_2(\mathbb{K})^{(-)}$, in the variables $Y$, when the base field is infinite is given by the following polynomials:
\[
[y_{i_1}, y_{i_2}, \ldots, y_{i_n}],
\]
where $n\geq 1$ and $i_1>i_2\leq\cdots\leq i_n$.
\end{remark}
Henceforth, we fix the basis of $\mathrm{UT}_2^{(-)}$ that was described in \Cref{basisUT2infinite}. Thus, a commutator of type $2$ has the form
\[c= [y_j, y_1^{(p_1)},\ldots, y_k^{(p_k)},\ldots y_j^{(p_j-1)},\ldots, y_n^{(p_n)}, z, y_1^{(q_1)},\ldots, y_n^{(q_n)}], \]
where $\deg_{y_s} m=p_s+q_s$ for $1\leq s\leq n$, $k=\min \{s\mid p_s\neq 0, 1\leq s\leq n-1 \}$ and $j>k$. We associate the following finite sequence 
\begin{equation}
S_c= ( k,j, ((p_1, k_1),\ldots, (p_n,k_n))).
\end{equation}
Note that $S_c\in\mathbb{N}\times\mathbb{N}\times D(\mathbb{N}_0 ^2)$. Now, define
\[\mathcal{S}=\{S_c\mid \textrm{$c$ is a commutator of type 2}\}.\]
Let $c_1$ and $c_2$ be commutators of type 2 and write
\[S_{c_1}=( k,j, ((p_1, k_1),\ldots, (p_n,k_n))), \]
\[S_{c_2}=( k',j', ((p'_1, k'_1),\ldots, (p'_m,k'_m))) .\]
We say that $S_{c_1}\preceq S_{c_2}$, if there exists an injective map $\psi:\mathbb{N}\to\mathbb{N}$ such that:  
\begin{itemize}
    \item[(i)] $\psi$ preserves the order, that is, if $u\leq v$ then $\psi(u)\leq \psi(v)$,
    \item[(ii)] $\psi(n)\leq m$,
    \item[(iii)] $\psi(k)=k'$, $\psi(j)=j'$,
    \item[(iv)] $p_i\leq p'_{\psi(i)}$ and $q_i\leq q'_{\psi(i)}$, for every $i\in \{1,\ldots,n\}$.
\end{itemize}
\begin{lemma}[{\cite[Theorem 5.13]{bahturin}}]\label{Sequenceorderemaining}
$(S,\preceq)$ is a partially well-ordered set.  
\end{lemma}

\begin{lemma}\label{consquenciaremaining}
Let $c= [y_j, y_1^{(p_1)},\ldots y_j^{(p_j-1)},\ldots, y_n^{(p_n)}, z, y_1^{(q_1)},\ldots, y_n^{(q_n)}]$, then
\[f=[y_j, y_1^{(p_1)},\ldots,y_j^{(p_j-1)},\ldots, y_n^{(p_n)}, y,z, y_1^{(q_1)},\ldots, y_n^{(q_n)}]\in\langle c\rangle^{T_G}+ \mathrm{Id}_{G} (\mathrm{UT}_3^{(-)},\Gamma(g,1)).\] 
\end{lemma}
\begin{proof}
Substituting $z\mapsto [z,y]$ on the commutator $c$, we have that 
\begin{equation*}
\begin{split}
& -[y_j, y_1^{(p_1)},\ldots,y_j^{(p_j-1)},\ldots ,y_n^{(p_n)}, [z, y], y_1^{(q_1)},\ldots, y_n^{(q_n)}] \\
& + 
[y_j, y_1^{(p_1)},\ldots,y_j^{(p_j-1)}, \ldots, y_n^{(p_n)}, z, y_1^{(q_1)},\ldots, y_n^{(q_n)}, y] +\mathrm{Id}_{G} (\mathrm{UT}_3^{(-)},\Gamma(g,1))\\
& = f + \mathrm{Id}_{G} (\mathrm{UT}_3^{(-)},\Gamma(g,1))
\end{split}
\end{equation*}
and the result follows.
\end{proof}
\begin{lemma}\label{oderSc1Sc2}
Let $c_1$ and $c_2$ be commutators of type $2$. If $S_{c_1}\preceq S_{c_2}$, then $c_2\in \langle c_1\rangle^{T_G}+ \mathrm{Id}_{G} (\mathrm{UT}_3^{(-)},\Gamma(g,1))$.
\end{lemma}
\begin{proof}
It follows from \Cref{consquenciaremaining} and the same idea as in the proof of \Cref{oderseqfe1e2}.
\end{proof}

\begin{definition}\label{orderlinearmonomiosremaining}
Consider two commutators of type 2 with the same multidegree
\[c_1=  [y_j, y_1^{(p_1)},\ldots, y_k^{(p_k)},\ldots y_j^{(p_j-1)},\ldots, y_n^{(p_n)}, z, y_1^{(q_1)},\ldots, y_n^{(q_n)}],\]
\[c_2= [y_{j'}, y_1^{(p'_1)},\ldots, y_{k'}^{(p'_{k'})},\ldots y_{j'}^{(p'_{j'}-1)},\ldots, y_n^{(p'_n)}, z, y_1^{(q'_1)},\ldots, y_n^{(q'_n)}].\]
We say that $c_1<' c_2$ if either
\begin{itemize}
\item[(a)] $k<k'$, or
\item[(b)] $k=k'$ and $j<j'$, or
\item[(c)] $k=k'$, $j=j'$ and $(q'_1,\ldots, q'_n)<_{l} (q_1,\ldots, q_n)$.
\end{itemize}
Here, $<_l$ is the lexicographic order on $\mathbb{N}_0^n$.
\end{definition}
Observe that $<'$ is a total ordering on the set of commutator of type 2 that have the same multidegree.
\begin{definition}
Let
\[f=\sum\limits_{i=1}^r \alpha_i c_i(z,y_1,\ldots, y_n),\]
where $\alpha_i\in\mathbb{K}\setminus\{0\}$ and $c_i$ is a commutator of type $2$, for $i\in\{1,\cdots, r\}$. The leading monomial of $f$ with respect to the order $<'$ is
\[\mu(f)= \max_{<'} \{c_i\mid 1\leq i\leq r\}.\]
\end{definition}
\begin{lemma}\label{ordenlinerremaining}
 Let $c_1$, $c_2$ be two commutator of type $2$ with the same multidegree, and write
\[c_1=  [y_j, y_1^{(p_1)},\ldots,y_j^{(p_j-1)},\ldots, y_n^{(p_n)}, z, y_1^{(q_1)},\ldots, y_n^{(q_n)}],\]
\[c_2= [y_{j'}, y_1^{(p'_1)},\ldots,y_{j'}^{(p'_{j'}-1)},\ldots, y_n^{(p'_n)}, z, y_1^{(q'_1)},\ldots, y_n^{(q'_n)}].\]
Assume that $c_1<' c_2$. Then
\renewcommand{\labelenumi}{(\roman{enumi})}
\begin{enumerate}
    \item \begin{equation*}
     \begin{split}   
    [y_j, y_1^{(p_1)},\ldots, y_j^{(p_j-1)},\ldots, y_n^{(p_n)}, z, y_1^{(q_1)},\ldots, y_i^{(q_i+1)},\ldots, y_n^{(q_n)}] < ' \\
    [y_{j'}, y_1^{(p'_1)},\ldots,y_{j'}^{(p'_{j'}-1)},\ldots, y_n^{(p'_n)}, z, y_1^{(q'_1)},\ldots, y_i^{(q'_i+1)} \ldots, y_n^{(q'_n)}],
    \end{split}
    \end{equation*}
for $i\in\{1,\ldots, n\}$.
\item  \begin{equation*}
     \begin{split}   
    [y_j, y_1^{(p_1)},\ldots, y_j^{(p_j-1)},\ldots, y_n^{(p_n)}, z, y_1^{(q_1)},\ldots, y_n^{(q_n)}, y_{n+1}] <' \\
    [y_{j'}, y_1^{(p'_1)},\ldots,  y_{j'}^{(p'_{j'}-1)},\ldots, y_n^{(p'_n)}, z, y_1^{(q'_1)},\ldots, y_n^{(q'_n)}, y_{n+1}].
    \end{split}
    \end{equation*}
\item  Consider 
\begin{equation*}
\begin{split}
c_1' & = -[y_j, y_1^{(p_1)},\ldots,  y_t^{(p_t+1)},\ldots, y_n^{(p_n)}, z, y_1^{(q_1)},\ldots, y_t^{(q_t)},\ldots, y_n^{(q_n)}]\\
& + [y_j, y_1^{(p_1)},\ldots,y_t^{(p_t)},\ldots, y_n^{(p_n)}, z, y_1^{(q_1)},\ldots, y_{t}^{(q_t+1)},\ldots ,y_n^{(q_n)}],
\end{split}
\end{equation*}
\begin{equation*}
\begin{split}
c_2'& = -[y_{j'}, y_1^{(p'_1)},\ldots, y_{t}^{(p'_t+1)} ,\ldots, y_n^{(p'_n)}, z, y_1^{(q'_1)},\ldots, y_t^{(q_t')},\ldots,y_n^{(q'_n)}]\\
& +[y_{j'}, y_1^{(p'_1)},\ldots, y_{t}^{(p'_t)} ,\ldots, y_n^{(p'_n)}, z, y_1^{(q'_1)},\ldots,y_{t}^{(q'_t+1)},\ldots,y_n^{(q'_n)}],
\end{split}  
\end{equation*}
where $k'\leq t\leq n$.  Then
\[\mu(c_1')= [y_j, y_1^{(p_1)},\ldots, y_t^{(p_t+1)},\ldots, y_n^{(p_n)}, z, y_1^{(q_1)},\ldots, y_t^{(q_t)},\ldots, y_n^{(q_n)}] ,\]
\[\mu(c_2')= [y_{j'}, y_1^{(p'_1)},\ldots,  y_{t}^{(p'_t+1)} ,\ldots, y_n^{(p'_n)}, z, y_1^{(q'_1)},\ldots,y_t^{(q'_t)},\ldots, y_n^{(q'_n)}],\]
and $\mu(c_1')<'\mu(c_2')$.
\item  
 Let
\begin{equation*}
\begin{split}
c_1'' & = -[y_j, y_1^{(p_1)},\ldots, y_n^{(p_n)}, y_{n+1}, z, y_1^{(q_1)}, \ldots, y_n^{(q_n)}]\\
& + [y_j, y_1^{(p_1)},\ldots, y_n^{(p_n)}, z, y_1^{(q_1)},\ldots ,y_n^{(q_n)}, y_{n+1}]
\end{split}
\end{equation*}
\begin{equation*}
\begin{split}
c_2''& = -[y_{j'}, y_1^{(p'_1)},\ldots, y_n^{(p'_n)}, y_{n+1}, z, y_1^{(q'_1)},\ldots, y_n^{(q'_n)}];\\
& +[y_{j'}, y_1^{(p'_1)},\ldots,y_n^{(p'_n)}, z, y_1^{(q'_1)},\ldots,y_n^{(q'_n)}, y_{n+1}].
\end{split}  
\end{equation*}
Then
\[\mu(c_1'')= [y_j, y_1^{(p_1)},\ldots, y_n^{(p_n)}, y_{n+1}, z, y_1^{(q_1)},\ldots, y_t^{(q_t)},\ldots, y_n^{(q_n)}],\]
\[\mu(c_2'')= [y_{j'}, y_1^{(p'_1)},\ldots, y_n^{(p'_n)}, y_{n+1}, z, y_1^{(q'_1)},\ldots,y_t^{(q'_t)},\ldots, y_n^{(q'_n)}], \]
and $\mu(c_1'')<'\mu(c_2'')$.
\end{enumerate}
\end{lemma}
\begin{proof}
 The proof is straightforward.
\end{proof}
\begin{lemma}\label{consecompatibleremaining}
Let $f$, $h\in\mathcal{H}^{(2)}$ and suppose that $S_{\mu(f)}\preceq S_{\mu(h)}$. Then, there exists $h'\in\mathcal{H}^{(2)}$ such that  $h'\in\langle f\rangle^{T_G}+\mathrm{Id}_G(\mathrm{UT}_3^{(-)},\Gamma(g,1))$ and $\mu(h)=\mu(h')$.
\end{lemma}
\begin{proof}
The proof is similar to \Cref{Consequenceuniversal}, taking a linear combination of evaluations followed by a product of variables of trivial degree instead of a single one; and replacing \Cref{oderseqfe1e2} by  \Cref{oderSc1Sc2} and \Cref{compativel} by \Cref{ordenlinerremaining}.
\end{proof}
\begin{lemma}\label{finitoH2}
Let $J$ be a $T_G$-ideal such that  $\mathrm{Id}_G(\mathrm{UT}_3^{(-)},\Gamma(g,1))\subseteq J$. Consider the set
\[\widetilde{\mathcal{H}}^{(2)}=\mathcal{H}^{(2)}\cap J.\]
Then, there exists a finite subset $\widetilde{H}_0\subseteq \widetilde{\mathcal{H}}^{(2)}$ such that
\[\widetilde{\mathcal{H}}^{(2)}\subseteq \langle \widetilde{H}_0\rangle^{T_G} + \mathrm{Id}_G(\mathrm{UT}_3^{(-)},\Gamma(g,1)). \]
\end{lemma}
\begin{proof}
The result follows from the ideas developed in the proof of \Cref{finitoe1e2} and from \Cref{Sequenceorderemaining,oderSc1Sc2,consecompatibleremaining}.
\end{proof}
Let $f\in\mathcal{H}$ and write
\[f= \alpha A(f) + g, \]
where $\alpha\in\mathbb{K}\setminus\{0\}$, $A(f)=[z, y_1^{(k_1)},\ldots, y_n^{(k_n)}]$ and $g\in\mathcal{H}^{(2)}$. Then, for each $f$, we associate the finite sequence $C_{f}=(k_1,\ldots, k_n)$. Let 
\[\mathcal{C}=\{ C_{f}\mid f\in\mathcal{H}\}.\]
Since $\mathcal{C}\subseteq D(\mathbb{N}_0)$ (see \Cref{Sequenceorder}), we induce the ordering of $D(\mathbb{N}_0)$ on $\mathcal{C}$. Hence, we have:
\begin{lemma}\label{seqsimplesremaining}
$(\mathcal{C},\preceq ')$ is a partially well-ordered set.\qed
\end{lemma}

\begin{lemma}\label{consequenciastipo1remaining}
Let $f, g\in\mathcal{H}$ and suppose that $C_{f}\preceq ' C_{g}$. Then, there exists $h\in\mathcal{H}$ such that  \[h\in\langle f\rangle^{T_G}+ \mathrm{Id}_G(\mathrm{UT}_3^{(-)},\Gamma(g,1))\] and $A(g)=A(h)$.
\end{lemma}
\begin{proof}
Suppose that
\[C_{f}=( k_1,\ldots, k_n),\quad C_{g}=(k'_1,\ldots, k'_m).\]
By hypothesis, there exists an order preserving injection $\psi:\mathbb{N}\to\mathbb{N}$ such that $\psi(n)\leq m$ and $k_i\leq k'_{\psi(i)}$, for every $i\in \{1,\ldots,n\}$.

Define $s_i:=\psi(i)$, for $i\in\{1,\ldots,n\}$. Then, the substitution $y_i\mapsto y_{s_i}$ followed by multiplication by $y$'s on the commutator $A(f)$ gives
\[m=[z, y_{i_1}^{(k_1)},\ldots, y_{i_n}^{(k_n)}, y_{i_1}^{(k_{i_1}'-k_1)},\ldots, y_{i_n}^{(k_{i_n}'-k_n)}, y_{j_{1}}^{(k'_{j_1})},\ldots,y_{j_{m-n}}^{(k'_{j_{m-n}})}], \]
where $\{j_1,\ldots, j_{m-n}\}=\{1\,\ldots,m\}\setminus\{i_1,\ldots, i_n\}$.
From \Cref{conjuntogeradorremaining}, we have 
\[m+  \mathrm{Id}_G(\mathrm{UT}_3^{(-)},\Gamma(g,1))= A(g)+ h' + \mathrm{Id}_G(\mathrm{UT}_3^{(-)}\Gamma(g,1)),\]
where $h'\in \mathcal{H}^{(2)}$. Repeating the same substitution and successive multiplication by variables $y$'s on the polynomial $f$, we obtain a polynomial $f'\in\langle f\rangle^{T_G}$ such that
\[f' + \mathrm{Id}_G(\mathrm{UT}_3^{(-)}\Gamma(g,1))=\gamma  A(g)+ h'' + \mathrm{Id}_G(\mathrm{UT}_3^{(-)}\Gamma(g,1)) ,\] 
where $\gamma\in\mathbb{K}\setminus\{0\}$ and $h''\in\mathcal{H}^2$. The result follows.
\end{proof}
\begin{lemma}\label{finitoH}
Let $J$ be a $T_G$-ideal such that  $\mathrm{Id}_G(\mathrm{UT}_3^{(-)},\Gamma(g,1))\subseteq J$. Consider the sets
$\widetilde{\mathcal{H}}^{(2)}=\mathcal{H}^{(2)}\cap J$ and $\widetilde{\mathcal{H}}=\mathcal{H} \cap J$.
Then there exists a finite subset $\widehat{H}\subseteq \widetilde{\mathcal{H}}$ such that
\[\widetilde{\mathcal{H}}\subseteq \langle \widehat{H}\rangle^{T_G} + \widetilde{\mathcal{H}}^{(2)}+\mathrm{Id}_G(\mathrm{UT}_3^{(-)},\Gamma(g,1)). \]
\end{lemma}
\begin{proof}
 Consider the sets
 \[\mathcal{A}_J=\{ A(f)\mid f\in \widetilde{\mathcal{H}} \},\quad \mathcal{C}_J=\{C_{f}\mid f\in\mathcal{A} \}.\]
 By \Cref{seqsimplesremaining}, there exist $f_1$,\dots, $f_t\in \widetilde{\mathcal{H}}$ such that $C_{f_1},\ldots, C_{f_t}$ are minimal elements of $\mathcal{C}$. Define $\widehat{H}=\{f_1,\ldots, f_t\}$
 and let $g\in\widetilde{\mathcal{H}}$. Then, there exists $1\leq i\leq t$ such that $C_{f_i}\preceq 'C_{A(g)}$. Thus, by  \Cref{consequenciastipo1remaining}, there exists $h\in \widetilde{\mathcal{H}}$ such that $h\in\langle f_i\rangle^{T_G}+ \mathrm{Id}_G(\mathrm{UT}_3^{(-)},\Gamma(g,1))$ and $A(h)=A(g)$. Without loss of generality, we can suppose that the coefficient of $A(g)$ is the same in the polynomials $f$ and $g$. Therefore, $h-g\in \widetilde{\mathcal{H}}^{(2)}$.
 Hence,
 \[g\in\langle f_i\rangle^{T_G}+\widetilde{\mathcal{H}}^{(2)}+\mathrm{Id}_G(\mathrm{UT}_3^{(-)},\Gamma(g,1))\subseteq  \langle \widehat{H}\rangle^{T_G} + \widetilde{\mathcal{H}}^{(2)}+\mathrm{Id}_G(\mathrm{UT}_3^{(-)},\Gamma(g,1))  \]
\end{proof}
\begin{lemma}\label{finitoH12}
Let $J$ be a $T_G$-ideal such that  $\mathrm{Id}_G(\mathrm{UT}_3^{(-)},\Gamma(g,1))\subseteq J$. Consider the set
$\widetilde{\mathcal{H}}^{(1,2)}=\mathcal{H}^{(1,2)}\cap J$. Then, there exists a finite subset $H^{(1,2)}\subseteq \widetilde{\mathcal{H}}^{(1,2)}$ such that
\[\mathcal{H}^{(1,2)}\subseteq \langle H^{(1,2)}\rangle^{T_G}+\mathrm{Id}_G(\mathrm{UT}_3^{(-)},\Gamma(g,1)). \]
\end{lemma}
\begin{proof}
The result follows from \Cref{finitoH2,finitoH}.  
\end{proof}
\begin{theorem}\label{remainingSpecht}
Let $\mathbb{K}$ be an infinite field, $G$ an abelian group, $g\in G\setminus\{1\}$ and $J$ be a $T_G$-ideal such that $\mathrm{Id}_G(\mathrm{UT}_3^{(-)}(\mathbb{K}),\Gamma(g,1))\subseteq J$. Then, there exists a finite subset $S\subseteq\mathcal{L}(X^G)$ such that $J=\langle S\rangle^{T_G}+\mathrm{Id}_G(\mathrm{UT}_3^{(-)},\Gamma(g,1))$.
\end{theorem}
\begin{proof}
Since the base field $\mathbb{K}$ is infinite, we have that 
\[J= \langle \mathcal{Y}\cup \widetilde{\mathcal{H}}^{(1,2)}\rangle^{T_{G}}+ \mathrm{Id}_G(\mathrm{UT}_3^{(-)},\Gamma(g,g^{-1})), \] 
where
\[\mathcal{Y}=\{\textrm{polynomials of trivial type}\}\cap J,\]
\[\widetilde{\mathcal{H}}^{(1,2)}=\mathcal{H}^{(1,2)}\cap J.\]
By \cite[Theorem 5.16]{bahturin} and \Cref{finitoH12}, there exist finite subsets $Y_0\subseteq\mathcal{Y} $ and $H^{(1,2)}\subseteq \widetilde{\mathcal{H}}^{(1,2)} $ such that
\[\mathcal{Y}\subseteq \langle Y_0\rangle^{T_G}+ \mathrm{Id}_G(\mathrm{UT}_3^{(-)},\Gamma(g,1)),\]
\[\widetilde{\mathcal{H}}^{(1,2)}\subseteq \langle H^{(1,2)}\rangle^{T_G}+ \mathrm{Id}_G(\mathrm{UT}_3^{(-)},\Gamma(g,1)).\]
Thus $J= \langle Y_0\rangle^{T_G}+ \langle H^{(1,2)}\rangle^{T_G}+ \mathrm{Id}_G(\mathrm{UT}_3^{(-)},\Gamma(g,1))$. The result is proved.
\end{proof}

\end{document}